\newif \ifpreprint
\preprinttrue

\ifpreprint
\documentclass[a4paper,10pt]{article}
\else
\documentclass{m2an}
\fi

\ifpreprint
\usepackage{geometry}
\usepackage{amsmath,amsthm}

\usepackage[symbol]{footmisc}
\else
\fi

\usepackage[breaklinks,bookmarks=false]{hyperref}
\hypersetup{colorlinks,linkcolor=blue,citecolor=blue,urlcolor=blue,plainpages=false}

\usepackage{pgfplots}
\usepackage{amssymb}
\usepackage{subcaption}

\newcommand{\plotpin}[5]{node[pos=#1,pin={[pin edge={#5},pin distance=#2]#3:{#4}}] {}}

\newenvironment{plotpage}[1]
{
\begin{minipage}{\figwidth\linewidth}
\begin{tikzpicture}
\begin{axis}
[
	width=\linewidth,
	x dir=reverse,
	xmode=log,
	ymode=log,
	xmax  = 12,
	xmin  = 0.03,
	#1
]
}
{
\end{axis}
\end{tikzpicture}
\end{minipage}
}

\newsavebox{\plottedpagetitle}
\newenvironment{plotpagetitle}[2]
{
\sbox{\plottedpagetitle}{#2}
\begin{minipage}{\figwidth\linewidth}
\begin{tikzpicture}
\begin{axis}
[
	width=\linewidth,
	x dir=reverse,
	xmode=log,
	ymode=log,
	xmax  = 12,
	xmin  = 0.03,
	#1
]
}
{
\end{axis}
\end{tikzpicture}
\subcaption{\usebox{\plottedpagetitle}}
\end{minipage}
}

\newcommand{\plotRE}[1]
{
\plot[dashed,red,mark=square*] table
[
	x expr=20/\thisrow{N},
	y expr=\thisrow{R}/(\thisrow{err_rcst_t}^2+\thisrow{err_rcst_g}^2)^0.5)
] {#1}
}

\newcommand{\plotME}[1]
{
\plot[dotted,blue,mark=triangle*] table
[
	x expr=20/\thisrow{N},
	y expr=(\thisrow{mis_tg}/\thisrow{rho}/(\thisrow{err_rcst_t}^2+\thisrow{err_rcst_g}^2)^0.5)
] {#1}
}

\newcommand{\plotMRE}[1]
{
\plot[solid,black,mark=*] table
[
	x expr=20/\thisrow{N},
	y expr=(\thisrow{R}^2+(20/\thisrow{rho}^2)*\thisrow{mis_tg}^2)^0.5
	      /(\thisrow{err_rcst_t}^2+\thisrow{err_rcst_g}^2)^0.5)
] {#1}
}

\newcommand{\plotCFLE}[2]
{
\plot[#2] table
[
	x expr=20/\thisrow{N},
	y expr=(\thisrow{err_rcst_t}^2+\thisrow{err_rcst_g}^2)^0.5)
] {#1}
}

\newcommand{\plotCFLMRE}[2]
{
\plot[#2] table
[
	x expr=20/\thisrow{N},
	y expr=(\thisrow{R}^2+(20/\thisrow{rho}^2)*\thisrow{mis_tg}^2)^0.5
	      /(\thisrow{err_rcst_t}^2+\thisrow{err_rcst_g}^2)^0.5)
] {#1}
}

\newcommand{\plotE}[1]
{
\plot[solid,black,mark=*] table
[
	x expr=20/\thisrow{N},
	y expr=(\thisrow{err_rcst_t}^2+\thisrow{err_rcst_g}^2)^0.5)
] {#1}
}

\newcommand{\plotM}[1]
{
\plot[dotted,blue,mark=triangle*] table
[
	x expr=20/\thisrow{N},
	y expr=\thisrow{mis_tg}/\thisrow{rho}
]
{#1}
}

\newcommand{\plotR}[1]
{
\plot[dashed,red,mark=square*] table
[
	x expr=20/\thisrow{N},
	y=R
]
{#1}
}

\newcommand{\plotF}[1]
{
\plot[solid,black,mark=o] table
[
	x expr=20/\thisrow{N},
	y=etaf
] {#1}
}

\newcommand{\ploteUv}[1]
{
\plot[solid,black,mark=*] table [x expr=20/\thisrow{N},y=err_node_v] {#1}
}

\newcommand{\ploteuv}[1]
{
\plot[dashed,red,mark=*] table [x expr=20/\thisrow{N},y=err_poly_v] {#1}
}

\newcommand{\plotewv}[1]
{
\plot[dotted,blue,mark=triangle*] table [x expr=20/\thisrow{N},y=err_rcst_v] {#1}
}

\newcommand{\ploteUx}[1]
{
\plot[solid,black,mark=square*] table [x expr=20/\thisrow{N},y=err_node_g] {#1}
}

\newcommand{\ploteux}[1]
{
\plot[dashed,red,mark=*] table [x expr=20/\thisrow{N},y=err_poly_g] {#1}
}

\newcommand{\plotewx}[1]
{
\plot[dotted,blue,mark=triangle*] table [x expr=20/\thisrow{N},y=err_rcst_g] {#1}
}

\newcommand{\ploteUt}[1]
{
\plot[solid,black,mark=square*] table [x expr=20/\thisrow{N},y=err_node_t] {#1}
}

\newcommand{\ploteut}[1]
{
\plot[dashed,red,mark=*] table [x expr=20/\thisrow{N},y=err_poly_t] {#1}
}

\newcommand{\plotewt}[1]
{
\plot[dotted,blue,mark=triangle*] table [x expr=20/\thisrow{N},y=err_rcst_t] {#1}
}

\newcommand{\effplot}[1]
{
\plotRE{#1}  node[pos=0.1,pin=90:{$R$}] {};
\plotME{#1}  node[pos=0.1,pin=0:{$M$}] {};
\plotMRE{#1} node[pos=0.5,pin=90:{$\eta$}] {};
}

\newcommand{\cvplot}[1]
{
\plotE{#1} node[pos=0.1,pin=90:{$\CE$}] {};
\plotM{#1} node[pos=0.1,pin=90:{$M$}] {};
\plotR{#1} node[pos=0.1,pin=90:{$R$}] {};
\plotF{#1} node[pos=0.1,pin=90:{$\eta_f$}] {};
}

\newcommand{\effminipage}[1]
{
\begin{minipage}{\figwidth\linewidth}
\begin{tikzpicture}
\begin{axis}
[
	width = \textwidth,
	x dir = reverse,
	xmode = log,
	ymode = log
]

\effplot{#1}

\end{axis}
\end{tikzpicture}
\end{minipage}}

\newcommand{\cvminipage}[1]
{
\begin{minipage}{\figwidth\linewidth}
\begin{tikzpicture}
\begin{axis}
[
	width = \textwidth,
	x dir = reverse,
	xmode = log,
	ymode = log,
	xmax  = 12,
	xmin  = 0.03
]

\cvplot{#1}

\end{axis}
\end{tikzpicture}
\end{minipage}}

\newcommand{\SlopeTriangle}[6]
{
    \pgfplotsextra
    {
        \pgfkeysgetvalue{/pgfplots/xmin}{\xmin}
        \pgfkeysgetvalue{/pgfplots/xmax}{\xmax}
        \pgfkeysgetvalue{/pgfplots/ymin}{\ymin}
        \pgfkeysgetvalue{/pgfplots/ymax}{\ymax}

        \pgfmathsetmacro{\xArel}{#1}
        \pgfmathsetmacro{\yArel}{#3}
        \pgfmathsetmacro{\xBrel}{#1-#2}
        \pgfmathsetmacro{\yBrel}{\yArel}
        \pgfmathsetmacro{\xCrel}{\xArel}
        \pgfmathsetmacro{\lnxB}{\xmin*(1-(#1-#2))+\xmax*(#1-#2)} % in [xmin,xmax].
        \pgfmathsetmacro{\lnxA}{\xmin*(1-#1)+\xmax*#1} % in [xmin,xmax].
        \pgfmathsetmacro{\lnyA}{\ymin*(1-#3)+\ymax*#3} % in [ymin,ymax].
        \pgfmathsetmacro{\lnyC}{\lnyA+#4*(\lnxA-\lnxB)}
        \pgfmathsetmacro{\yCrel}{\lnyC-\ymin)/(\ymax-\ymin)} % THE IMPROVED EXPRESSION WITHOUT 'DIMENSION TOO LARGE' ERROR.

        \coordinate (A) at (rel axis cs:\xArel,\yArel);
        \coordinate (B) at (rel axis cs:\xBrel,\yBrel);
        \coordinate (C) at (rel axis cs:\xCrel,\yCrel);

        \draw[#6]   (A)-- node[anchor=north] {#5}
                    (B)--
                    (C)--
                    cycle;
    }
}

\input def.sty

\newtheorem{theorem}{Theorem}
\newtheorem{lemma}[theorem]{Lemma}
\newtheorem{corollary}[theorem]{Corollary}
\newtheorem{remark}[theorem]{Remark}

\numberwithin{equation}{section}
\numberwithin{theorem}{section}

\ifpreprint
\newcommand{\figwidth}{.45}
\else
\newcommand{\figwidth}{.35}
\fi

\ifpreprint
\newcommand{\rev}[1]{#1}
\else
\newcommand{\rev}[1]{{\color{red} #1}}
\fi

\begin{document}

\ifpreprint
\title{Damped energy-norm a posteriori error estimates using $C^2$-reconstructions for the fully discrete wave equation \rev{with the leapfrog scheme} \footnotemark[1]}

\author{T. Chaumont-Frelet\footnotemark[2], A. Ern\footnotemark[3]}

\footnotetext[2]{Inria Univ. Lille and Laboratoire Paul Painlev\'e, 59655 Villeneuve-d'Ascq, France}
\footnotetext[3]{CERMICS, Ecole nationale des ponts et chauss\'ees, IP Paris, 6 \& 8 avenue B.~Pascal, 77455 Marne-la-Vall\'{e}e, France and Inria, 48 Rue Barrault, 75647 Paris, France}
\footnotetext[1]{This work was supported by the ANR JCJC project APOWA (research grant ANR-23-CE40-0019-01).}

\date{\today}

\maketitle

\begin{abstract}
We derive a posteriori error estimates for the the scalar wave equation discretized
in space by continuous finite elements and in time by the explicit leapfrog scheme.
Our analysis combines the idea of invoking extra time-regularity for the right-hand
side, as previously introduced in the space semi-discrete setting, with a novel, 
piecewise quartic, globally twice-differentiable time-reconstruction of the fully
discrete solution. Our main results show that the proposed estimator is reliable and
efficient in a damped energy norm. These properties are illustrated in a series of
numerical examples.

\vspace{.5cm}
\end{abstract}
\else
\title{Damped energy-norm a posteriori error estimates using $C^2$-reconstructions for the fully discrete wave equation \rev{with the leapfrog scheme}}

\author{T. Chaumont-Frelet}\address{Inria Univ. Lille and Laboratoire Paul Painlev\'e, 59655 Villeneuve-d'Ascq, France}
\author{A. Ern}\address{CERMICS, Ecole nationale des ponts et chauss\'ees, IP Paris, 6 \& 8 avenue B.~Pascal, 77455 Marne-la-Vall\'{e}e, France and Inria, 48 Rue Barrault, 75647 Paris, France}

\thanks{This work was supported by the ANR JCJC project APOWA (research grant ANR-23-CE40-0019-01).}



\begin{abstract}
We derive a posteriori error estimates for the the scalar wave equation discretized
in space by continuous finite elements and in time by the explicit leapfrog scheme.
Our analysis combines the idea of invoking extra time-regularity for the right-hand
side, as previously introduced in the space semi-discrete setting, with a novel, 
piecewise quartic, globally twice-differentiable time-reconstruction of the fully
discrete solution. Our main results show that the proposed estimator is reliable and
efficient in a damped energy norm. These properties are illustrated in a series of
numerical examples.
\end{abstract}

\maketitle
\fi

\section{Introduction}

Given an open, bounded, Lipschitz polyhedron $\Omega\subset \Real^d$, $d\ge1$, 
with boundary $\partial\Omega$,
the time interval $J:=[0,+\infty)$, and a source 
term $f:J\times \Omega\to \Real$, the scalar wave equation consists in finding 
$u:J\times \Omega\to \Real$ such that
\begin{subequations}
\label{eq:BVP}
\begin{alignat}{2}
&\ddot u - \Delta u = f&\quad&\text{in $J\times \Omega$}, \label{eq:PDE} \\
&u=0&\quad&\text{on $J\times \partial\Omega$}, \label{eq:BC} \\
&u|_{t=0}=\dot u|_{t=0}=0&\quad&\text{in $\Omega$}. \label{eq:IC}
\end{alignat}
\end{subequations}
The homogeneous Dirichlet condition \eqref{eq:BC} is considered 
for the sake of simplicity, and non-homogeneous coefficients in space could be
considered in \eqref{eq:PDE}. Instead, the zero initial conditions 
\eqref{eq:IC} play a role in our analysis. Moreover, the source term $f$ is 
assumed to be smooth in time and supported away from zero.
Notice that the time-smoothness assumption on $f$ does not preclude 
dealing with minimal space-regularity in domains generating
corner or edge singularities.

The model problem \eqref{eq:BVP} is of relevance in
many engineering applications, so that its numerical discretization
has been extensively developed and analyzed. 
Here, we shall focus on the method of lines in its simplest form, 
where continuous finite elements are employed for the space discretization,
combined with the (explicit) leapfrog scheme as time-marching scheme. This is one of the
most frequently used methods to discretize~\eqref{eq:BVP} owing to its 
computational efficiency with appropriate mass lumping techniques.
Recall, in particular, that stability of explicit
time-integrators for the wave equation is typically achieved under
a mild CFL condition of the form $\tau \lesssim h$,
which is often required for accuracy reasons anyway 
(here, $h$ denotes the mesh size and $\tau$ the time step).

In this work, we are interested in rigorously estimating the discretization error
using an a posteriori error estimator. We aim at deriving both upper and lower bounds
on the error by the estimator, i.e., reliability and efficiency properties.
Perhaps surprisingly, only few works address this question in the literature,
as compared to the vast number of references dealing with elliptic and parabolic 
problems (see, e.g.,
\cite{ern_smears_vohralik_2017a,
ern_vohralik_2015a,
makridakis_nochetto_2003a,
verfurth_2003a,
verfurth_2013a}
and the references therein).
\rev{One important challenge} for deriving reliable
and efficient a posteriori error estimates for the wave equation
is the lack of an inf-sup stability framework
in natural norms \cite[Theorem 4.2.23]{zank_2019a}.
Advances towards inf-sup stable variational formulations of
\eqref{eq:BVP} have been recently reported in
\cite{bignardi_moiola_2023a,fuhrer_gonzalez_karkulik_2023a,zank_2019a}. 
Moreover, in the context of boundary integral equations,
we refer the reader to \cite{hoonhout_loscher_steinbach_urzuatorres_2023a} for a
least-squares approach, and to \cite{gimperlein_ozdermir_stark_stephan_2020a}
for residual a posteriori error estimates.
However, all the above formulations lead to fully coupled space-time discretizations,
whereas, in the present work, our objective is instead to cover the method of lines.

The a posteriori error analysis of the wave equation discretized with
the method of lines has been previously addressed in
\cite{bernardi_suli_2005a,%
georgoulis_lakkis_makridakis_2013a,%
georgoulis_lakkis_makridakis_virtanen_2016a,%
gorynina_lozinski_picasso_2019a}. On the one hand, an implicit time discretization is
considered in
\cite{bernardi_suli_2005a,%
georgoulis_lakkis_makridakis_2013a,%
gorynina_lozinski_picasso_2019a}, either using a second-order backward differentiation
formula for the second-order time derivative or a Newmark-type scheme. The error measure
is the $H^1(J;H^{-1}(\Omega))\cap L^2(J;L^2(\Omega))$-norm in \cite{bernardi_suli_2005a},
the $L^\infty(J;L^2(\Omega))$-norm in \cite{georgoulis_lakkis_makridakis_2013a}, and the
energy-norm ($H^1(J;L^2(\Omega))\cap L^2(J;H^1_0(\Omega))$) in
\cite{gorynina_lozinski_picasso_2019a}. Among these three works,
only \cite{bernardi_suli_2005a} also derives error lower bounds, but the efficiency result
is somewhat polluted by the presence of additional terms involving the error energy-norm. On
the other hand, the (explicit) leapfrog scheme is considered in
\cite{georgoulis_lakkis_makridakis_virtanen_2016a}\rev{, but only 
in a time semi-discrete setting. 
However, for the wave equation, time semi-discrete schemes cannot propagate compactly 
supported initial data beyond the initial support, and explicit time semi-discrete 
solutions must sit in the domain of iterated powers of the Laplacian operator.} 

The above discussion clearly indicates that a gap remains in the literature concerning the 
a posteriori error analysis of the fully discrete wave equation using the (explicit) leapfrog
scheme. The goal of the present work is to partly fill this gap. To this purpose, we rely on the
approach recently introduced in \cite{Theo:23} in the space semi-discrete setting to bypass the
lack of inf-sup stability of the wave equation. 
There are two key ideas in \cite{Theo:23}. The first one 
is to (abstractly) work with the Laplace transform and distinguish 
low- and high-frequency components of the error. The former can be controlled 
by invoking a duality argument, and the latter by invoking the time smoothness 
of the source term $f$ \rev{(the fact that extra time-regularity is required
is somehow related to the lack of inf-sup stability)}.
The second idea is to bound the space semi-discrete error, $e$, using 
the following damped energy-norm:
\begin{equation} \label{eq:damped_energy}
\calE_\rho^2(e) := 
\int_0^{+\infty} \Big\{ \|\dot e(t)\|_\Omega^2 + \|\nabla e(t)\|_\Omega^2 \Big\} e^{-2\rho t}dt,
\end{equation} 
where the damping parameter $\rho>0$ scales as the reciprocal of a time and can be
chosen as small as desired. In practice, one is typically interested in the solution
up to some finite time-horizon $T_*$, and one then sets $\rho$ to be, e.g., the reciprocal of 
some multiple of $T_*$. \rev{Two important remarks are in order. First, assuming that
the source term satisfies $f\in L^2(J;L^2(\Omega))$, classical arguments (briefly recalled 
in Remark~\ref{rem:bnd_engy_u}) show that
the undamped energy of the exact solution, $\polE_u(t) := \frac12\|\dot u(t)\|_\Omega^2
+ \frac12 \|\nabla u(t)\|_\Omega^2$, grows at most linearly with $t$. Under the slightly 
tighter assumption $f\in H^1(J;L^2(\Omega))$ and the above CFL restriction on the time step, 
we verify in Remark~\ref{rem:bnd_engy_wht}
that the same property holds true for the undamped energy of the (time reconstructed) 
discrete solution (see below for its precise definition). 
This justifies neglecting the tail of the time integral 
in~\eqref{eq:damped_energy} in practical computations. 
Second, although the leapfrog scheme enjoys uniform-in-time
stability, this does not lead to uniform-in-time error estimates that decay to zero with the 
mesh size (even in the classical a priori error analysis based on undamped energy arguments). 
Thus, the introduction of a time horizon, here by means of the damping parameter $\rho$ in~\eqref{eq:damped_energy}, is a natural ingredient in the a priori and a posteriori error analysis.}

The main step forward accomplished herein is to extend \cite{Theo:23}
to the fully discrete setting \rev{using the leapfrog scheme in time
with a constant time step and a fixed space discretization}. 
This step is by no means straightforward. The key idea is 
to introduce suitable time-reconstructed functions from the sequence of fully discrete 
solutions produced by the leapfrog scheme at the discrete time nodes. One crucial difficulty
is that reformulating the leapfrog scheme in a time-functional setting requires
introducing two time reconstructions, one which is piecewise quadratic in time and 
globally of class $C^0$, say $\uht$, and the other which is piecewise quartic in time 
and globally of class $C^2$, say $\wht$. The idea of introducing a time reconstruction
from values produced at the discrete time nodes by a time-marching scheme
is already known in the context of a posteriori error analysis 
\cite{AkMaN:09}. In the context of the wave equation for instance, a piecewise cubic 
reconstruction, which is globally of class $C^1$ in time, is introduced in 
\cite{georgoulis_lakkis_makridakis_2013a}. However, the $C^2$ time reconstruction
introduced herein seems, to our knowledge, a novel idea in the analysis
of the wave equation.

Using the above two reconstructions $\uht$ and $\wht$, the leapfrog scheme
can be rewritten as follows: For all $t\in J$ and all $v_h\in V_h$, 
\begin{equation} \label{eq:leapfrog_intro}
(\ddwht(t),v_h)_\Omega + (\nabla \uht(t),\nabla v_h)_\Omega = (f_\tau(t),v_h)_\Omega.
\end{equation}
Here, $V_h\subset V:=H^1_0(\Omega)$ is a finite element space and $f_\tau$ a suitable
approximation in time of the source term $f$ (precise notation is introduced 
in Section~\ref{section_cont_disc_setting}).
Defining the fully discrete error 
\begin{equation} \label{eq:def_e}
e:=u-\wht,
\end{equation} 
the main consequence of~\eqref{eq:leapfrog_intro}
is that the error equation takes the following form: For all $t\in J$ and all $v\in V$,
\begin{equation} \label{eq:error_eq_intro}
(\ddot e(t),v)_\Omega + (\nabla e(t),\nabla v)_\Omega = (\eta_f(t),v)_\Omega + (\nabla \deltaht(t),\nabla v)_\Omega,
\end{equation}
where $\eta_f:=f-f_\tau$ represents a data oscillation term 
(which can be shown to decay at higher order in time than the error itself) and
\begin{equation} \label{eq:def_delta_intro}
\deltaht := \uht - \wht.
\end{equation}
The last term on the right-hand side of~\eqref{eq:error_eq_intro}
leads to a lack of Galerkin orthogonality in the 
error equation. The consequence of this fact is that norms of 
$\deltaht$ appear in both upper and lower bounds on the error. Notice 
that $\deltaht$ is fully computable and can be viewed as 
a time discretization error estimator. As confirmed
by our numerical experiments, the contribution of $\deltaht$ to the a posteriori error
estimator can be made small by decreasing the time step. We also emphasize that the lack
of Galerkin orthogonality is also present when dealing with the heat equation discretized
using continuous finite elements and an implicit time-scheme (e.g., backward Euler),
as already highlighted in \cite{ern_smears_vohralik_2017a}.

Let us briefly outline our main results. Let $\calE_\rho^{2}(e)$ 
be the damped energy norm of the fully discrete error $e=u-\wht$ 
(see~\eqref{eq:damped_energy}. Our main result in Theorem \ref{thm:upper}
below states that, up to higher-order terms (h.o.t.) \rev{as the mesh size and the time step tend to zero},
\begin{equation} \label{eq:outline_thm12}
\calE_\rho^2(e)
\leq
\int_{0}^{+\infty} \Big \{
\eta_h^2(t) + \frac{\rev{20}}{\rho^2}\|\nabla \ddeltaht(t)\|_\Omega^2
\Big \}
e^{-2\rho t} dt
+
\text{h.o.t.}
\end{equation}
with
\begin{equation}
\eta_h(t) := \| \ddwht(t) - \Delta \rev{\wht(t)} - f_\tau(t)\|_{H^{-1}(\Omega)}.
\end{equation}
Notice that the \rev{the right-hand side of~\eqref{eq:outline_thm12}}
% error estimator, $\Lambda_\rho$, defined in~\eqref{eq:outline_thm12}
is the sum of two terms, which
may be linked, respectively, to the space and the time discretization errors\rev{, plus
higher-order terms further discussed in Remark~\ref{rem:hot_up}. We observe that 
all the higher-order terms are either explicitly computable or admit a computable upper bound 
(which is not of higher-order, but may be useful to certify the error)}. Moreover, 
since $\ddwht - \Delta \uht - f_\tau$ enjoys a Galerkin orthogonality property
at all times $t\in J$
owing to~\eqref{eq:leapfrog_intro}, the a posteriori estimator $\eta_h(t)$ can be
bounded \rev{by a triangle inequality producing the term $\deltaht$ and}
any technique available in the elliptic context, e.g., of residual-type
or based on flux equilibration. 
A converse bound on $\eta_h$ is established in Theorem \ref{thm:lower} under
a CFL constraint on the time step, namely (up to a constant independent of the
mesh size $h$, time step $\tau$, and damping parameter $\rho$)
\begin{equation} \label{eq:low_intro}
\int_{0}^{+\infty} \eta_h^2(t) e^{-2\rho t} dt
\lesssim
\calE_{\rho}^2(e) +
\int_{0}^{+\infty} \|\nabla \deltaht(t)\|^2 e^{-2\rho t} dt
+ \text{h.o.t.},
\end{equation}
\rev{with higher-order terms further discussed in Remark~\ref{rem:hot_low}}.
Thus, if the term
\begin{equation}
\calD_{\rho}^2 := \int_{0}^{+\infty} \Big \{
\|\nabla \deltaht(t)\|^2 + \frac{1}{\rho^2}\|\nabla \ddeltaht(t)\|_\Omega^2 \Big\} 
e^{-2\rho t} dt
\end{equation}
is added to the error measure and to the estimator by setting
$\widetilde \calE_\rho^2(e) := \calE_\rho^2(e) + \calD_\rho^2$ and
$\widetilde \Lambda_\rho^2 := \int_{0}^{+\infty} \eta_h^2(t) e^{-2\rho t} dt + \calD_\rho^2$, 
one obtains the following reliability and efficiency result:
\begin{equation} \label{eq:equiv_intro}
\widetilde \Lambda_{\rho}^2 - \text{h.o.t.}
\lesssim
\widetilde \calE_{\rho}^2(e)
\lesssim
\widetilde \Lambda_{\rho}^2 + \text{h.o.t.}
\end{equation}
Incorporating a fully computable term like $\calD_\rho$ in the error 
and the estimator is standard in other
contexts, such as discontinuous Galerkin discretizations of elliptic problems
\cite{KarPa:03}. \rev{Notice that we prove in Lemma~\ref{lem:bnd_eta} below
that $\calD_\rho$ decreases optimally in time (i.e., as $\tau^2$).}
In our numerical experiments on the wave equation,
we observe that
\rev{the right-hand side of~\eqref{eq:outline_thm12}} % $\Lambda_{\rho}$
controls $\calE_{\rho}$ by a factor of at most $10$.

\rev{We notice that the present limitation of a constant time step and a fixed space
discretization is not easy to lift. This is not specific to the present a posteriori
analysis since the obstruction already arises in the a priori error analysis. For instance,
stability issues may arise if the time step varies, as discussed in \cite{Skeel:93}.
Moreover, mesh changes can, in principle, be taken into account in the error upper 
bound by including in the estimator the energy variations induced by the mesh changes,
but the question of incorporating these changes in an error lower bound remains open.}

The remainder of the paper is organized as follows. We make the continuous and
discrete settings precise in Sections \ref{section_cont_disc_setting}. 
We present the time reconstructions
in Section \ref{section_time_reconstructions}, where we also investigate their accuracy.
Sections \ref{section_upper_bound} and \ref{section_lower_bound} are, respectively,
dedicated to establishing the upper and lower bounds on the error. 
We present numerical examples in Section \ref{section_numerical_results}.
In \rev{Section}~\ref{section_stability_leapfrog}, we derive some stability results on
the leapfrog scheme in the damped energy norm. These results are useful
in establishing the error lower bounds, but we believe they are of independent interest.
\rev{For completeness}, we \rev{also} derive an a priori error
estimate on the exact solution in the damped energy norm\rev{. For a priori
error estimates in the classical energy norm, we refer the reader, e.g., to
\cite{Joly:03} and the references therein.}

\section{Continuous and discrete settings}
\label{section_cont_disc_setting}

In this section, we present the continuous problem and state its basic stability properties
in the damped energy norm. Next, we present the discrete setting based on continuous
finite elements for the space discretization and the leapfrog scheme for the time discretization.
Finally, we recall the classical CFL stability condition for the leapfrog scheme,
and we introduce an approximation factor quantifying how well the finite element space
approximates some dual solution which is used to establish the error upper bound.

\subsection{Time and frequency domains}

We use standard notation for the Lebesgue, Sobolev, and Bochner--Sobolev spaces.
In particular, $(\cdot,\cdot)_\Omega$ denotes the inner product in $L^2(\Omega)$
and $\|\cdot\|_\Omega$ the corresponding norm, and we employ the same notation
for vector-valued functions with all components in $L^2(\Omega)$. 
For any Banach space $Y$ composed of functions defined over $\Omega$, 
and for any integer $r\ge0$, we set $C^r_{\rm b}(J;Y) := \{v\in C^r(J;Y) \st v^{(r)}\in L^\infty(J;Y)\}$
and $C^2_{0,}(J;Y):=\{v\in C^2(J;Y) \st v(0)=\dot v(0)=0\}$. 
Moreover, for every function $\phi \in H^1(J;L^2(\Omega))\cap L^2(J;V)$
with $V:=H^1_0(\Omega)$, we define its damped energy norm as
\begin{equation} \label{eq:def_calE_rho}
\calE_\rho^2(\phi) := \int_0^{+\infty} \Big\{ \|\dot \phi(t)\|_\Omega^2 + \|\nabla \phi(t)\|_\Omega^2 \Big\}
e^{-2\rho t}dt.
\end{equation}

We assume that the source term satisfies $f\in \rev{H^1(J;L^2(\Omega)) \cap L^\infty(J;L^2(\Omega))}$ \rev{so that, in particular, $f\in C^0_{\rm b}(J;L^2(\Omega))$.} Moreover, we assume that 
$f$ is supported away from zero.
It is then natural to seek for a strong solution
$u\in C^2_{0,}(J;L^2(\Omega)) \cap C^0(J;V)$ such that
\begin{equation} \label{eq:PDE_weak}
(\ddot u(t),v)_\Omega + (\nabla u(t),\nabla v)_\Omega = (f(t),v)_\Omega,
\quad \forall t\in J, \; \forall v\in V.
\end{equation}
Notice that the boundary condition \eqref{eq:BC} is encoded in the fact that $u\in C^0(J;V)$,
and the initial conditions \eqref{eq:IC} are encoded in the fact that
$u\in C^2_{0,}(J;L^2(\Omega))$.

A convenient way to handle the damped energy-norm introduced in~\eqref{eq:damped_energy}
is to work in the frequency domain. To this purpose, we
consider the Laplace transform
\begin{equation}
\hv(s) := \int_0^{+\infty} v(t)e^{-st}dt,
\qquad s:=\rho + \sfi \nu \in \polC, \; \Re(s):=\rho>0, \; \Im(s):=\nu, 
\end{equation} 
for any function $v\in L^\infty(J;Y)$ so that the integral is properly defined.
We observe that the damping parameter $\rho>0$ introduced in the error measure determines 
the real part of the complex frequency $s$. A key property of the Laplace transform is that
\begin{equation} \label{eq:identity}
\int_0^{+\infty} \|v(t)\|_Y^2 e^{-2\rho t}dt
=
\int_{\rho-\sfi\infty}^{\rho+\sfi\infty} \|\hv(s)\|_Y^2 ds.
\end{equation}

For all $s \in \polC$, we define the sesquilinear form $b_s$ on $\hV\times \hV$
with $\hV:=H^1_0(\Omega;\polC)$ such that
\begin{equation} \label{eq:def_b}
b_s(\hphi,\hv) := s^2(\hphi,\hv)_\Omega + (\nabla \hphi,\nabla \hv)_\Omega.
\end{equation}
In the frequency domain, \eqref{eq:PDE_weak} can be rewritten as
\begin{equation} \label{eq:PDE_freq}
b_s(\hu(s),\hv) = (\hf(s),\hv)_\Omega \quad \forall s\in\polC, \; \forall \hv\in\hV.
\end{equation}

\subsection{A priori estimates on the exact solution}

Defining the ``frequency-domain energy norm'' as
\begin{equation}
\tnorm{\hv}^2 := |s|^2 \|\hv\|_\Omega^2 + \|\nabla \hv\|_\Omega^2 \quad \forall \hv\in \hV,
\end{equation}
the key stability property of the sesquilinear form $b_s$ defined in~\eqref{eq:def_b} is
\begin{equation} \label{eq:stability}
\rho \tnorm{\hv}^2 = \Re\big( b_s(\hv,s\hv) \big) \quad \forall \hv\in \hV.
\end{equation}

\begin{lemma}[A priori estimate] \label{lem:a_priori}
Assume that $u \in C^1_{\rm b}(J;L^2(\Omega))\cap C^0_{\rm b}(J;V)$.
The following holds:
\begin{equation} \label{eq:key_stab1}
\calE_\rho^2(u) \le \frac{1}{\rho^2} \int_0^{+\infty} \|f(t)\|_\Omega^2 e^{-2\rho t}dt.
\end{equation}
\end{lemma}

\begin{proof}
Owing to the stability property~\eqref{eq:stability}, we infer that, for all $s\in\polC$,
\[
\rho \tnorm{\hu(s)}^2 = \Re\big( b_s(\hu(s),s\hu(s))\big)
= \Re (\hf(s),s\hu(s))_\Omega \le \|\hf(s)\|_\Omega \tnorm{\hu(s)}.
\]
Hence, we have
\begin{equation} \label{eq:key_stab_freq}
\tnorm{\hu(s)} \le \frac{1}{\rho} \|\hf(s)\|_\Omega \quad \forall s\in\polC.
\end{equation} 
The assertion follows from $s\hu(s)=\hat{\dot u}(s)$ and~\eqref{eq:identity}.
\end{proof} 

\begin{remark}[\rev{Bound on undamped energy}] \label{rem:bnd_engy_u}
\rev{The undamped energy $\polE_u(t) := \frac12\|\dot u(t)\|_\Omega^2
+ \frac12 \|\nabla u(t)\|_\Omega^2$ satisfies the identity $\dot{\polE}_u(t) =
(f(t),\dot u(t))_\Omega$ at all times. Fix any time $t>0$. Then
$\dot{\polE}_u(s) \le \frac{t}{2}\|f(s)\|_\Omega^2 + \frac{1}{t}\polE_u(s)$ 
for all $s\in (0,t)$, and a Gronwall-like
argument readily gives $\polE_u(t) \le \frac{e}{2} t\int_0^t \|f(s)\|_\Omega^2ds$. 
This shows that $\polE_u(t)$ grows at most linearly in $t$ if $f\in L^2(J;L^2(\Omega))$.}
\end{remark}

\subsection{Discrete problem}

The space discretization is realized by means of a conforming finite 
element method (FEM). We assume that $\Omega$ is a polyhedron and
consider an affine simplicial mesh $\calT_h$ that covers $\Omega$ exactly.
The subscript $h$ refers to the mesh size. Let $V_h\subset V$ be the continuous
FEM space built using Lagrange finite elements of degree $k\ge1$. 
The time discretization is realized by considering an increasing sequence of discrete
time nodes $(t^n)_{n\in\polN}$ (conventionally, $0\in\polN$). 
\rev{As motivated in the introduction, we assume a constant time-step $\tau$ so that $t^n=n\tau$ for all $n\in\polN$.}
Recalling that the damping parameter $\rho$ is such that $\rho T\ge 1$, where $T$
is some observation time-scale, we make in what follows the following mild assumption:
\begin{equation}\label{eq:ass_rho_dt}
\rho \tau \le 1.
\end{equation}
The second-order time derivative in the scalar wave equation is discretized by means of the leapfrog (central finite difference) scheme. The fully discrete wave equation consists in finding the sequence $(\sfU^n)_{n\in\polN}\subset V_h^\polN$, i.e., $\sfU^n\in V_h$ for all $n\in\polN$, such that 
\begin{equation} \label{eq:leapfrog}
\frac{1}{\tau^{2}}(\sfUnp-2\sfUn+\sfUnm,v_h)_\Omega 
+ (\nabla \sfUn,\nabla v_h)_\Omega = (f^n,v_h)_\Omega \quad
\forall n\ge1,\; \forall v_h\in V_h,
\end{equation}
with $f^n:=f(t^n)$ and the initial conditions $\sfU^0=\sfU^1=0$. (Recall that $f(0)=0$ by assumption.)

\begin{remark}[Time-horizon]
In practice, one fixes a finite time-horizon $T_{\star}=N\tau$ for some positive integer
$N$ and computes only the first $N$ steps of the scheme \eqref{eq:leapfrog}.
\end{remark}

\subsection{CFL condition}

It is well-known that the leapfrog scheme is conditionally stable under a CFL condition.
This condition is not needed in our analysis to derive the error upper bound, but it is
invoked in the error lower bound. 
To state the CFL restriction, one introduces on $V_h\times V_h$ the bilinear form
\begin{equation} \label{eq:def_mht}
m_{h\tau}(v_h,w_h) := (v_h,w_h)_\Omega - \tau^2 (\nabla v_h,\nabla w_h)_\Omega.
\end{equation}
Then, for all $\mu_0\in (0,1)$, there exists $\tau^*(\mu_0)>0$ such that, whenever 
$\tau \in (0,\tau^*(\mu_0)]$, the following holds:
\begin{equation} \label{eq:mht_L2}
\mu_0 (v_h,w_h)_\Omega \le m_{h\tau}(v_h,w_h) \le (v_h,w_h)_\Omega
\quad \forall v_h,w_h\in V_h.
\end{equation}
(Notice that the second bound is trivial.) One can take $\mu_0=\frac12$ in what follows to fix the ideas. Invoking an inverse inequality, one readily
shows that $\tau^*(\mu_0)\le Ch_{\min}$, where $h_{\min}$ is the smallest diameter of
a mesh cell and the constant $C$ depends on $\mu_0$ and the shape-regularity parameter 
of the mesh $\calT_h$. We leave the dependence on the latter parameter implicit
and write the CFL condition in the form
\begin{equation} \label{eq:CFL}
\tau \le C(\mu_0)h_{\min},
\end{equation} 
for some positive constant $C(\mu_0)$ such that \eqref{eq:mht_L2} holds true for some $\mu_0\in (0,1)$.

\subsection{Approximation factor}
\label{sec:app_factor}

As in \cite{Theo:23}, the derivation of the error upper bound involves a duality
argument. To this purpose, for all $\hg\in L^2(\Omega;\polC)$, we consider
the solution $\hchi_g \in \hV$ to the adjoint problem $b_s(\hw,\hchi_g) = |s|^2
(\hw,\hg)_\Omega$ for all $\hw\in \hV$. We define the approximation factor
\begin{equation} \label{eq:def_gamma_s}
\gamma_s(h) := \sup_{\substack{\hg\in L^2(\Omega;\polC) \\ |s|\|\hg\|_\Omega=1}} \min_{\hv_h\in \hV_h} \|\nabla(\hchi_g-\hv_h)\|_\Omega,
\end{equation}
where $\hV_h$ is the complex-valued version of the finite element space $V_h$ introduced above. 

\begin{lemma}[Bound on approximation factor] \label{lem:bnd_gamma}
The following holds:
\begin{equation} \label{eq:bnd_gamma}
\rev{\gamma}_{s}(h)
\le
\min \bigg\{
\frac{|s|}{\rho},
C_{\mathrm{app}} C_{\mathrm{ell}} h^\theta \ell_\Omega^{1-\theta} |s|
\Big( 1+ \frac{|s|}{\rho} \Big)
\bigg\},
\end{equation}
where $\ell_\Omega:=\operatorname{diam}(\Omega)$ is a global length scale introduced for dimensional consistency,
$C_{\mathrm{app}}$ is related to the approximation properties of the finite element space $\hV_h$, and $\theta\in (\frac12,1]$ and $C_{\mathrm{ell}}$ are related to the elliptic regularity shift in $\Omega$.
\end{lemma}

\begin{proof}
Let $\hg\in L^2(\Omega;\polC)$ be s.t.~$|s|\|\hg\|_\Omega=1$.
Invoking~\eqref{eq:stability} and taking $\hv_h=0$ readily shows that 
\[
\gamma_s(h)\le \|\nabla\hchi_g\|_\Omega \le \tnorm{\hchi_g} \le \frac{|s|}{\rho}.
\]
Moreover, invoking elliptic regularity in Lipschitz polyhedra (see \cite[p.~158]{Dauge:88}),
we infer that there are $\theta\in (\frac12,1]$ and $C_{\mathrm{ell}}>0$ such that
$\|\hchi_g\|_{H^{1+\theta}(\Omega)} \le C_{\mathrm{ell}} \ell_\Omega^2 \|\Delta \hchi_g\|_\Omega$.
Using the approximation properties of the finite element space $\hV_h$ to bound the minimum over
$\hv_h$ in \eqref{eq:def_gamma_s}, this gives
\begin{align*}
\|\nabla(\hchi_g-\hv_h)\|_\Omega & \le C_{\mathrm{app}} h^\theta \ell_\Omega^{-1-\theta}\|\hchi_g\|_{H^{1+\theta}(\Omega)} \le C_{\mathrm{app}} C_{\mathrm{ell}} h^\theta \ell_\Omega^{1-\theta} \|\Delta \hchi_g\|_\Omega  \\
& \le C_{\mathrm{app}} C_{\mathrm{ell}} h^\theta \ell_\Omega^{1-\theta} |s| \big( 1+ \tfrac{|s|}{\rho} \big), 
\end{align*}
where the last bound follows from $\Delta \hchi_g=s^2\hchi_g-|s|^2\hg$, $|s| \|\hg\|_\Omega=1$,
and $|s| \|\hchi_g\|_\Omega \le \tnorm{\hchi_g} \le \frac{|s|}{\rho}$. 
\end{proof}

For any cutoff frequency $\omega>0$ and any damping parameter $\rho>0$, we set
\begin{equation} \label{eq:def_gamma_rho_omega}
\gamma_{\rho,\omega}(h) := \max_{\substack{s=\rho + \sfi\nu\\ |\nu|\le \omega}} \gamma_s(h).
\end{equation}
Bounds on $\gamma_{\rho,\omega}(h)$ are readily derived from Lemma~\ref{lem:bnd_gamma}. 
For instance, $\gamma_{\rho,\omega}(h) \le (1+(\frac{\omega}{\rho})^2)^{\frac12}$. 
More importantly, fixing $\omega>0$ and $\rho>0$, we infer 
that $\gamma_{\rho,\omega}(h)\to0$ as $h\to0$.

\section{Time reconstructions for the leapfrog scheme}
\label{section_time_reconstructions}

In this section, we introduce two time reconstructions defined on sequences in $V_h$. 
The two reconstructions satisfy an important commuting property with the second-order 
(discrete) time derivative and allow us to rewrite the leapfrog scheme in a time-functional setting.

For all $n\in\polN$, we consider the time interval $J_n:=[t^n,t^{n+1})$. Let $Y$ be a Banach space composed of functions defined over $\Omega$; typical examples include $Y=V_h$, $Y=V$ or $Y=L^2(\Omega)$. Given a polynomial degree $\ell\ge0$, we define the following broken polynomial space in time:
\begin{equation}
P^\ell(J_\tau;Y) := \{ \vht \in L^\infty(J;Y) \st \vht|_{J_n} \in \polP^\ell(J_n;Y) \ \forall n\in\polN \},
\end{equation}
where $\polP^\ell(J_n;Y)$ is composed of $Y$-valued time-polynomials of order at most $\ell$ restricted to $J_n$.

\subsection{Definitions and key properties}

Consider a sequence $\sfV:=(\sfV^n)_{n\in\polN}$ in $Y$. Henceforth, any such 
sequence is extended by zero for negative indices, i.e., we conventionally
set $\sfV^{-1}=\sfV^{-2}=\ldots:=0$. The first time reconstruction we consider is 
the function 
\begin{equation}
R(\sfV) \in P^2(J_\tau;Y) \cap C^0(J;Y),
\end{equation}
which is defined such that, for all $n\in\polN$ and all $t\in J_n$,
\begin{equation} \label{eq:def_R}
R(\sfV)(t) := \sfVn + \frac{\sfVnp-\sfVnm}{2\tau}(t-t^n) + \frac{\sfVnp-2\sfVn+\sfVnm}{\tau^2} \frac12(t-t^n)^2.
\end{equation}
Notice that $R(\sfV)$ is the restriction to $J_n$ of the Lagrange interpolate at the discrete time nodes $t^{n-1}$, $t^n$, $t^{n+1}$.
While the fact that $R(\sfV) \in P^2(J_\tau;Y)$ is obvious by construction, we now
justify the claim $R(\sfV) \in C^0(J;Y)$.

\begin{lemma}[Time-reconstruction] \label{lem:R_C0}
Let $R(\sfV)$ be defined by \eqref{eq:def_R}.
Then, $R(\sfV) \in C^0(J;Y)$.
\end{lemma}

\begin{proof}
For all $n\in \polN$, we observe that
\[
R(\sfV)(t^{n+1}_-)=\sfVn + \frac{\sfVnp-\sfVnm}{2} + \frac{\sfVnp-2\sfVn+\sfVnm}{2} = \sfVnp = R(\sfV)(t^{n+1}).
\]
This completes the proof.
\end{proof}

The second time reconstruction provides $C^2$-smoothness in time, and its construction hinges on piecewise quartic time-polynomials. Specifically, we consider the time-reconstructed function
\begin{equation}
L(\sfV) \in P^4(J_\tau;Y) \cap  C^2(J;Y),
\end{equation}
which is defined such that, for all $n\in\polN$ and all $t\in J_n$,
\begin{equation} \label{eq:def_L}
L(\sfV)(t) := \alpha^n + \beta^n(t-t^n) + \gamma^n \frac12(t-t^n)^2
+ \rev{\vartheta^n} \frac16(t-t^n)^3 + \epsilon^n \frac{1}{24}(t-t^n)^4,
\end{equation}
with the coefficients
\begin{subequations} \label{eq:def_L_bis} \begin{alignat}{2}
\alpha^n &= \frac{3\sfVnp+17\sfVn+5\sfVnm-\sfVnmm}{24}, &\quad
\beta^n &= \frac{5\sfVnp+3\sfVn-9\sfVnm+\sfVnmm}{12\tau}, \\
\gamma^n &= \frac{\sfVnp-2\sfVn+\sfVnm}{\tau^2}, &\quad
\rev{\vartheta^n} &= \frac{\sfVnpp-2\sfVnp+2\sfVnm-\sfVnmm}{2\tau^3}, \\
\epsilon^n &= \frac{\sfVnpp-4\sfVnp+6\sfVn-4\sfVnm+\sfVnmm}{\tau^4}. &&
\end{alignat} \end{subequations}
As above, the fact that $L(\sfV) \in P^4(J_\tau;Y)$ is obvious by construction, and we now
justify the claim $L(\sfV) \in C^2(J;Y)$.

\begin{lemma}[Smooth time-reconstruction] \label{lem:L_C2}
Let $L(\sfV)$ be defined by \eqref{eq:def_L}-\eqref{eq:def_L_bis}.
Then, $L(\sfV) \in C^2(J;Y)$.
\end{lemma}

\begin{proof}
Let us set $v:=L(\sfV)$.
For all $n\in \polN$, elementary manipulations show that
\begin{align*}
v(t^{n+1}_-) &=\alpha^n+\beta^n\tau+\gamma^n\frac12\tau^2+\rev{\vartheta^n}\frac16\tau^3
+ \epsilon^n\frac{1}{24}\tau^4 = \alpha^{n+1} = v(t^{n+1}), \\
\dot v(t^{n+1}_-) &=\beta^n+\gamma^n\tau+\rev{\vartheta^n}\frac12\tau^2
+ \epsilon^n\frac{1}{6}\tau^3 = \beta^{n+1} = \dot v(t^{n+1}), \\
\ddot v(t^{n+1}_-) &=\gamma^n+\rev{\vartheta^n}\tau
+ \epsilon^n\frac{1}{2}\tau^2 = \gamma^{n+1} = \ddot v(t^{n+1}).
\end{align*}
This concludes the proof.
\end{proof}

\begin{remark}[Consistency]
Elementary manipulations of Taylor polynomials show that in the case of a sequence
$\sfV$ such that $\sfVn=v(t^n)$ for some smooth function $v\in C^6(J;Y)$ supported
away from zero, the coefficients defining the time-reconstructed function $L(\sfV)$  
in~\eqref{eq:def_L_bis} are such that 
$\alpha^n = v(t^n)+O(\tau^2)$, $\beta^n=\dot v(t^n) + O(\tau^2)$,
$\gamma^n=\ddot v(t^n)+O(\tau^2)$, $\rev{\vartheta^n}=\dddot v(t^n)+O(\tau^2)$, and
$\epsilon^n=\ddddot v(t^n)+O(\tau^2)$. 
\end{remark}

Consider a sequence $\sfV:=(\sfV^n)_{n\in\polN}$ in $Y$ such that $\sfV^0=0$.
Define the sequence $(\rev{D^2_\tau}\sfV^n)_{n\in\polN}$ such that 
$\rev{D^2_\tau}\sfV^n := \frac{1}{\tau^2}\big( \sfV^{n+1}-2\sfV^n+\sfV^{n-1} \big)$,
for all $n\in\polN.$ (It is legitimate to set $\rev{D^2_\tau}\sfV^{-1}=0$ since 
$\sfV^0=0$.)

\begin{lemma}[Commuting with second-order time derivative] \label{lem:commut}
The following holds:
\begin{equation} \label{eq:commut}
\frac{d^2}{dt^2} L(\sfV)(t) = R(\rev{D^2_\tau} \rev{\sfV})(t) \quad \forall t\in J.
\end{equation}
\end{lemma}

\begin{proof}
Since $L(\sfV)$ is of class $C^2$ and $R(\rev{D^2_\tau} \rev{\sfV})$ is of class $C^0$ owing to
Lemma~\ref{lem:L_C2} and Lemma~\ref{lem:R_C0}, respectively, it suffices to establish
the identity~\eqref{eq:commut} for all $t\in J_n$ and all $n\in\polN$. 
Recalling the definitions in~\eqref{eq:def_L_bis}, elementary manipulations show that
the coefficients $\gamma^n$, $\rev{\vartheta^n}$, and $\epsilon^n$ defining $L(\sfV)$ satisfy
\[
\gamma^n = \rev{D^2_\tau} \sfV^n, \qquad \rev{\vartheta^n} = \frac{\rev{D^2_\tau} \sfV^{n+1}-\rev{D^2_\tau} \sfV^{n-1}}{2\tau},
\qquad \epsilon^n = \frac{\rev{D^2_\tau} \sfV^{n+1}-2\rev{D^2_\tau} \sfV^n+\rev{D^2_\tau} \sfV^{n-1}}{\tau^2}.
\]
As a consequence, we have 
\[
\frac{d^2}{dt^2} L(\sfV)(t) = \gamma^n + \rev{\vartheta^n}(t-t^n)+\epsilon^n\frac12(t-t^n)^2
= R(\rev{D^2_\tau}\sfV)(t) \quad \forall t\in J_n, \; \forall n\in\polN.
\]
This concludes the proof.
\end{proof}

\subsection{Rewriting of the leapfrog scheme}

The commuting property established in Lemma~\ref{lem:commut} allows us to rewrite the leapfrog scheme in a time-functional setting. 
Let us set 
\begin{equation} \label{eq:def_ftau}
f_\tau := R(\sfF)\in P^2(J_\tau;L^2(\Omega))\cap C^0(J;L^2(\Omega)) 
\quad\text{with}\quad \sfF^n:=f(t^n)\; \forall n\in\polN.
\end{equation} 
Notice that the sequence $(\sfF^n)_{n\in\polN}$ can indeed be extended by zero
for negative indices since $f$ is supported away from zero. We also set
\begin{equation} \label{eq:defs_uw}
\uht := R(\sfU) \in P^2(J_\tau;V_h)\cap C^0(J;V_h),
\qquad
\wht := L(\sfU) \in P^4(J_\tau;V_h)\cap C^2(J;V_h), 
\end{equation}
where the sequence $\sfU:=(\sfU^n)_{n \in \polN}$ solves the fully discrete scalar wave
equation~\eqref{eq:leapfrog} \rev{with $\sfU^0=\sfU^1=0$}. The sequence of accelerations is defined as
$\sfA:=(\sfA^n)_{n\in\polN}$ with
\begin{equation} \label{eq:def_sfA}
\sfA^n := \rev{D^2_\tau}\sfU^n = \frac{1}{\tau^2}(\sfUnp-2\sfUn+\sfUnm) \quad \forall n\in\polN.
\end{equation}
Notice that $\sfA^0=0$ by definition and that it is legitimate to set
$\sfA^{-1}=\sfA^{-2}=\ldots=0$.

\begin{lemma}[Time-reconstructed wave equation]
The following holds:
\begin{equation} \label{eq:time_leapfrog}
(\ddwht(t),v_h)_\Omega + (\nabla \uht(t),\nabla v_h)_\Omega = (f_\tau(t),v_h)_\Omega
\quad \forall t\in J, \; \forall v_h \in V_h.
\end{equation}
\end{lemma}

\begin{proof}
Since the fully discrete wave equation can be rewritten as
\[
(\sfA^n,v_h)_\Omega + (\nabla \sfU^n,\nabla v_h)_\Omega = (\sfF^n,v_h)_\Omega
\quad \forall n\in\polN, \; \forall v_h \in V_h,
\]
the claim follows by applying the time-reconstruction operator $R$ to this equation and invoking Lemma~\ref{lem:commut} which gives $R(\sfA)(t) = \frac{d^2}{dt^2}L(\sfU) = \ddwht(t)$.
\end{proof} 

\subsection{Stability and approximation properties}

In this section, we investigate some stability and approximation properties
of the time reconstruction operator $R$. In what follows, for positive real
numbers $A$ and $B$, we abbreviate as $A\lesssim B$ the inequality $A\le CB$
with a generic constant $C$ whose value can change at each occurrence as long
as it is independent of the time step, the mesh size, the damping parameter $\rho$,
and, whenever relevant, any function involved in the bound.

\begin{lemma}[Stability]
\label{lem:stab}
Let $\sfV:=(\sfV^n)_{n\in\polN}$ be a bounded sequence in $Y$ 
and set $\dot \sfV^{n+\frac12}:= \frac{1}{\tau}(\sfV^{n+1}-\sfV^n)$.
The following holds:
\begin{subequations} \begin{align} \label{eq:stab1}
\int_0^{+\infty} \|R(\sfV)\|_Y^2 e^{-2\rho t}dt & \lesssim \sum_{n\in\polN} \tau \|\sfV^n\|_Y^2 e^{-2\rho t^n}, \\
\int_0^{+\infty} \|\tfrac{d}{dt}R(\sfV)\|_Y^2 e^{-2\rho t}dt & \lesssim \sum_{n\in\polN} \tau \|\dot\sfV^{n+\frac12}\|_Y^2 e^{-2\rho t^n}. \label{eq:stab2}
\end{align} \end{subequations}
\end{lemma}

\begin{proof}
Invoking inverse inequalities in time shows that $\|R(\sfV)\|_{L^\infty(J_n;Y)} \lesssim \max_{\delta\in\{-1,0,1\}} \|\sfV^{n+\delta}\|_Y$. This gives 
\[
\int_0^{+\infty} \|R(\sfV)\|_Y^2 e^{-2\rho t}dt \lesssim \sum_{n\in\polN} \tau \max_{\delta\in\{-1,0,1\}} \|\sfV^{n+\delta}\|_Y^2 e^{-2\rho t^n},
\]
whence \eqref{eq:stab1} follows since $\rho\tau\le 1$ (see~\eqref{eq:ass_rho_dt}).
The proof of~\eqref{eq:stab2} is similar and uses that $\|\frac{d}{dt}R(\sfV)\|_{L^\infty(J_n;Y)} 
\lesssim \max(\|\dot\sfV^{n-\frac12}\|_Y,\|\dot\sfV^{n+\frac12}\|_Y)$.
\end{proof}

\begin{remark}[\rev{Bound on undamped energy for time-reconstructed solutions}] 
\label{rem:bnd_engy_wht}
\rev{For all $n\in\polN$, we define the midpoint state and velocity,
$\sfU\nmz := \frac12 (\sfU^{n+1}+\sfU^n)$ and 
$\dsfU\nmz := \frac{1}{\tau}(\sfU^{n+1}-\sfU^n)$, 
as well as the discrete undamped energy
$E_\sfU\nmz := \frac12 \|\dsfU\nmz\|_\Omega^2 + \frac12\|\nabla \sfU\nmz\|^2_\Omega$. 
Classical arguments for the leapfrog scheme
(see \textup{\cite{Joly:03}} and Section~\textup{\ref{sec:abstract}} for a global-in-time
bound on the damped energy)
show that, under the CFL condition~\eqref{eq:CFL}, we have
\begin{equation} \label{eq:bnd_undamped_disc_engy}
E_\sfU\nmz \lesssim t^n \sum_{m\in\{0{:}n\}} \tau\|\sfF^n\|_\Omega^2 \lesssim t^n \|f\|_{H^1_\tau(0,t^n;L^2(\Omega))}^2,
\end{equation}
with
\begin{equation*}
\|f\|_{H^1_\tau(0,t^n;L^2(\Omega))}^2
:=
\|f\|_{L^2(0,t^n;L^2(\Omega))}^2 + \tau^2 \|\dot f\|_{L^2(0,t^n;L^2(\Omega))}^2
\end{equation*}
and where the second bound follows from the embedding $H^1(J_n;L^2(\Omega)) \hookrightarrow C^0(\overline J_n;L^2(\Omega))$. Define now, for all $t\in J$, the undamped energies
$\polE_{\uht}(t):=\frac12\|\duht(t)\|_\Omega^2 + \frac12\|\nabla \uht(t)\|_\Omega^2$
and $\polE_{\wht}(t):=\frac12\|\dwht(t)\|_\Omega^2 + \frac12\|\nabla \wht(t)\|_\Omega^2$. 
Owing to~\eqref{eq:def_R}, we infer that, for all $n\in\polN$, 
\[
\|\duht\|_{L^\infty(J_n;L^2(\Omega))} \lesssim \max\big( \|\dsfU\nmz\|_\Omega,\|\dsfU\nmzm\|_\Omega\big),
\]
and
\begin{align*}
\|\nabla\uht\|_{L^\infty(J_n;L^2(\Omega))} &\lesssim \max\big( \|\nabla\sfU^{n+1}\|_\Omega,
\|\nabla\sfU^{n}\|_\Omega,\|\nabla\sfU^{n-1}\|_\Omega\big) \\
&\lesssim \max\big( \|\nabla\sfU\nmz\|_\Omega,
\|\nabla\sfU\nmzm\|_\Omega\big) 
+ \frac{\tau}{h} \max\big( \|\dsfU\nmz\|_\Omega,\|\dsfU\nmzm\|_\Omega\big),
\end{align*}
where the second bound follows from elementary algebraic manipulations and an inverse 
inequality in space. Combining these bounds and invoking the CFL condition readily shows that
$\|\polE_{\uht}\|_{L^\infty(J_n)} \lesssim t^n \|f\|_{H^1_\tau(0,t^n;L^2(\Omega))}^2$. Similar arguments
using~\eqref{eq:def_L}-\eqref{eq:def_L_bis} (omitted for brevity) show that
$\|\polE_{\wht}\|_{L^\infty(J_n)} \lesssim t^n \|f\|_{H^1_\tau(0,t^n;L^2(\Omega))}^2$. In conclusion,
the undamped energies of the time-reconstructed solutions, $\polE_{\uht}(t)$ and
$\polE_{\uht}(t)$, grow at most linearly in $t$ if $f\in H^1(J;L^2(\Omega))$.
}
\end{remark}

For a function $v\in C^0(J;Y)$ supported away from zero, we set $\sfV^n:=v(t^n)$ for all $n\in\polN$ (so that $\sfV^0=v(0)=0$ by assumption). We extend $v$ by zero for $t\le0$ and we set $\sfV^{-1}=\sfV^{-2}=\ldots:=0$. We now bound the approximation error $v-R(\sfV)$ assuming enough smoothness of $v$ in time. 

\begin{lemma}[Approximation] \label{lem:etaf}
The following holds for every function $v\in C^3_{\rm b}(J;Y)$ supported away from zero: 
\begin{subequations} \begin{align} 
\int_0^{+\infty} \|(v-R(\sfV))(t)\|_Y^2 e^{-2\rho t} dt & \lesssim \tau^6 \int_0^{+\infty} \|\dddot v(t)\|_Y^2 e^{-2\rho t} dt, \label{eq:est_etaf} \\
\int_0^{+\infty} \|\tfrac{d}{dt}(v-R(\sfV))(t)\|_Y^2 e^{-2\rho t} dt & \lesssim \tau^4 \int_0^{+\infty} \|\dddot v(t)\|_Y^2 e^{-2\rho t} dt. \label{eq:est_detaf}
\end{align} \end{subequations}
\end{lemma}

\begin{proof}
Let $v\in C^3_{\rm b}(J;Y)$ be supported away from zero
and set $\dot\sfV^n:=\dot v(t^n)$, $\ddot\sfV^n:=\ddot v(t^n)$ for all $n\in\polN$. 

(1) Proof of~\eqref{eq:est_etaf}. We observe that
\begin{equation} \label{eq:proof_etav0}
\int_0^{+\infty} \|(v-R(\sfV))(t)\|_Y^2 e^{-2\rho t} dt \le \sum_{n\in\polN} e^{-2\rho t^n} \int_{J_n} \|(v-R(\sfV))(t)\|_Y^2dt.
\end{equation}
Let $n\in\polN$ and $t\in J_n$. Using a third-order Taylor expansion of $v$ with exact remainder, we observe that
\begin{equation} \label{eq:proof_etav}
(v-R(\sfV))(t) = \rev{\Xi}^n_1(t-t^n) + \rev{\Xi}^n_2\frac12(t-t^n)^2 + \rev{\Xi}_3(t),
\end{equation}
with
\begin{align*}
\rev{\Xi}_1^n &:= \dot \sfV^n - \frac{\sfV^{n+1}-\sfV^{n-1}}{2\tau}, \quad
\rev{\Xi}^n_2 := \ddot \sfV^n - \frac{\sfV^{n+1}-2\sfV^n+\sfV^{n-1}}{\tau^2}, \quad
\rev{\Xi}_3(t) := \frac12 \int_{t^n}^t (t-s)^2 \dddot v(s) ds.
\end{align*}
Invoking first-order Taylor expansions with exact remainder gives
\begin{align*}
\sfV^{n+1} = \sfV^n + \tau \dot \sfV^n + \int_{t^n}^{t^{n+1}} (t^{n+1}-s)\ddot v(s)ds,
\qquad
\sfV^{n-1} = \sfV^n - \tau \dot \sfV^n + \int_{t^{n-1}}^{t^n} (s-t^{n-1})\ddot v(s)ds.
\end{align*}
We infer that
\[
\rev{\Xi}^n_1 = \frac12 \int_{t^n}^{t^{n+1}}\psi^n(s) \ddot v(s) ds - 
\frac12 \int_{t^{n-1}}^{t^n}\psi^n(s) \ddot v(s) ds = \frac12 \int_{t^n}^{t^{n+1}}\psi^n(s)
\big( \ddot v(s)-\ddot v(\tilde s) \big) ds,
\]
with $\tilde s:=2t^n-s$ and $\psi^n$ denotes the hat basis function in time having support in $[t^{n-1},t^{n+1}]$ and satisfying $\psi^n(t^n)=1$. Since $\|\ddot v(s)-\ddot v(\tilde s)\|_Y \le 
\int_{\tilde s}^s \|\dddot v(\sigma)\|_Y d\sigma \le \tau^{\frac12} \big( \int_{t^{n-1}}^{t^{n+1}} \|\dddot v(\sigma)\|_Y^2d\sigma\big)^{\frac12}$ for all $s\in [t^n,t^{n+1}]$, we infer that
\[
\|\rev{\Xi}^n_1\|_Y^2 \lesssim \tau^3 \int_{t^{n-1}}^{t^{n+1}} \|\dddot v(\sigma)\|_Y^2d\sigma.
\] 
Moreover, using second-order Taylor expansions with exact remainder gives
\[
\rev{\Xi}^n_2 = \frac12 \int_{t^n}^{t^{n+1}}\psi^n(s)^2
\dddot v(s)- \frac12 \int_{t^{n-1}}^{t^n}\psi^n(s)^2 \dddot v(\tilde s) ds, 
\]
so that
\[
\|\rev{\Xi}^n_2\|_Y^2 \lesssim \tau \int_{t^{n-1}}^{t^{n+1}} \|\dddot v(\sigma)\|_Y^2d\sigma.
\]
(Here, it is not necessary to invoke the fourth-order derivative of $v$ to gain an extra power of $\tau$.)
Finally, we have
\[
\|\rev{\Xi}_3\|_{C^0(J_n;Y)}^2 \lesssim \tau^5 \int_{t^n}^{t^{n+1}} \|\dddot v(\sigma)\|_Y^2d\sigma
\quad \forall t\in J_n.
\]
Combining~\eqref{eq:proof_etav} with the above bounds gives
\[
\int_{J_n} \|(v-R(\sfV))(t)\|_Y^2dt \lesssim \tau\big( \tau^2\|\rev{\Xi}_1^n\|_Y^2 +
\tau^4 \|\rev{\Xi}_2^n\|_Y^2 + \|\rev{\Xi}_3\|_{C^0(J_n;Y)}^2 \big) \lesssim
\tau^6 \int_{t^{n-1}}^{t^{n+1}} \|\dddot v(\sigma)\|_Y^2d\sigma.
\]
Using this estimate in~\eqref{eq:proof_etav0} gives
\[
\int_0^{+\infty} \|(v-R(\sfV))(t)\|_Y^2 e^{-2\rho t} dt \lesssim \tau^6 \sum_{n\in\polN} e^{-2\rho t^n}
\int_{t^{n-1}}^{t^{n+1}} \|\dddot v(t)\|_Y^2dt,
\]
and the assertion follows from 
\[
e^{-2\rho t^n} \int_{t^{n-1}}^{t^{n+1}} \|\dddot v(t)\|_Y^2dt \le
\int_{J_{n-1}} \|\dddot v(t)\|_Y^2 e^{-2\rho t}dt + e^2\int_{J_{n}} \|\dddot v(t)\|_Y^2e^{-2\rho t}dt,
\]
recalling that $\rho \tau\le 1$ owing to~\eqref{eq:ass_rho_dt}. 

(2) The proof of~\eqref{eq:est_detaf} is similar and is only sketched. We observe that for all $t\in J_n$,
\[
\frac{d}{dt}(v-R(\sfV))(t) = \rev{\Xi}^n_1 + \rev{\Xi}^n_2(t-t^n) + \rev{\Xi}_4(t),
\]
with $\rev{\Xi}_4(t):=\int_{t^n}^t (t-s)\dddot v(s)ds$. Using the above bounds on $\|\rev{\Xi}^n_1\|_Y$ and $\|\rev{\Xi}^n_2\|_Y$ together with $\|\rev{\Xi}_4\|_{C^0(J_n;Y)}^2\lesssim \tau^3 \int_{t^n}^{t^{n+1}} \|\ddot v(s)\|_Y^2ds$ readily proves~\eqref{eq:est_detaf}. 
\end{proof}

\subsection{A priori estimates on the time-reconstruction error}

The rewriting of the leapfrog scheme in a time-functional setting naturally
leads to the notion of time-reconstruction error defined as
\begin{equation} \label{eq:def_deltaht}
\deltaht(t) := \uht(t)-\wht(t)\quad \forall t\in J.
\end{equation}
The goal of this section is to derive some a priori estimates on the 
time-reconstruction error. It is natural to expect that this quantity
is second-order accurate in $\tau$. We now establish a more precise result
using the damped energy norm. Recall
that $\deltaht \in C^0(J;V_h)$ and $\wht\in C^2(J;V_h)$,
so that the weak time-derivatives 
$\ddeltaht$ and $\dddot{w}_{h\tau}$ can be evaluated by computing locally
the time derivative(s) in each time interval $J_n$. 

\begin{lemma}[Bound on time-reconstruction error] \label{lem:bnd_eta}
Assume that the sequence $(\sfU^n)_{n\in\polN}$ solves the leapfrog scheme~\eqref{eq:leapfrog}
with the initial conditions $\sfU^0=\sfU^1=0$. 
Let $\uht$ and $\wht$ be defined in~\eqref{eq:defs_uw}. 
Assume that $f\in C^3_{\rm b}(J;L^2(\Omega))$.
Assume the CFL condition~\eqref{eq:CFL}. The following holds:
\begin{subequations} \begin{align} 
\int_0^{+\infty} \|\ddeltaht(t)\|_\Omega^2 e^{-2\rho t}dt & \lesssim \frac{\tau^4}{\rho^2}
\int_0^{+\infty} \|\ddot f(t)\|_\Omega^2 e^{-2\rho t}dt, \label{eq:bnd_dot_eta} \\
\int_0^{+\infty} \|\nabla\deltaht(t)\|_\Omega^2  e^{-2\rho t}dt & \lesssim \frac{\tau^4}{\rho^2}
\int_0^{+\infty} \Big\{ \|\ddot f(t)\|_\Omega^2 + \tau^2 \|\dddot f(t)\|_\Omega^2 \Big\} 
e^{-2\rho t}dt, \label{eq:bnd_grad_eta} \\
\int_0^{+\infty} \|\nabla\ddeltaht(t)\|_\Omega^2 e^{-2\rho t}dt & \lesssim \frac{\tau^4}{\rho^2}
\int_0^{+\infty} \|\dddot f(t)\|_\Omega^2 e^{-2\rho t}dt. \label{eq:bnd_grad_dot_eta}
\end{align} 
In addition, we have
\begin{equation} \label{eq:bnd_dddwht}
\int_0^{+\infty} \|\dddot{w}_{h\tau}(t)\|_\Omega^2 e^{-2\rho t}dt
\lesssim \frac{1}{\rho^2}
\int_0^{+\infty} \|\ddot f(t)\|_\Omega^2 e^{-2\rho t}dt.
\end{equation} \end{subequations}
\end{lemma}

\begin{proof}
Recall the definition~\eqref{eq:def_sfA} of $\sfA^n$ and set
$\sfA\nmz := \frac12(\sfA^{n+1}+\sfA^n)$, 
$\dsfA\nmz := \frac{1}{\tau}(\sfA^{n+1}-\sfA^n)$ for all $n\ge -1$.
A direct computation
shows that for all $t\in J_n$ and all $n\in\polN$,
\begin{align*}
\tau^{-2} \deltaht(t) = {}& \frac{1}{12}(\sfA\nmzm+\tau\dsfA\nmzm) - \frac{\tau}{12} 
\dsfA\nmzm \zeta^n_1(t) + \frac{\tau}{2}(\dsfA\nmz+\dsfA\nmzm) \zeta^n_3(t) \\
&+ \tau (\dsfA\nmz-\dsfA\nmzm) \zeta^n_4(t),\\
\tau^{-2} \ddeltaht(t) = {}& 
- \frac{1}{12} \dsfA\nmzm + \frac{1}{2}(\dsfA\nmz+\dsfA\nmzm) \zeta^n_2(t)+ (\dsfA\nmz-\dsfA\nmzm) \zeta^n_3(t), \\
\dddot{w}_{h\tau}(t) = {}& \frac{1}{2}(\dsfA\nmz+\dsfA\nmzm) + (\dsfA\nmz-\dsfA\nmzm) \zeta^n_1(t),
\end{align*}
where we used the shorthand notation $\zeta^n_m(t):=\frac{1}{m!}\frac{1}{\tau^m}(t-t^n)^m$ for all $m\in\{0,\ldots,4\}$. Notice that all of these polynomials are bounded on $J_n$.
Hence, we have
\begin{align}
\int_0^{+\infty} \|\ddeltaht(t)\|_\Omega^2 e^{-2\rho t}dt & \lesssim  
\tau^4 \sum_{n\in\polN} \big( \|\dsfA\nmzm\|_\Omega^2 + \|\dsfA\nmz\|_\Omega^2\big) \int_{J_n} e^{-2\rho t}dt \nonumber  \\
& \le \tau^4 \sum_{n\in\polN} \tau \big( \|\dsfA\nmzm\|_\Omega^2 + \|\dsfA\nmz\|_\Omega^2\big) e^{-2\rho t^n} \nonumber \\
& \lesssim \tau^4 \sum_{n\in\polN} \tau \|\dsfA\nmz\|_\Omega^2 e^{-2\rho t^n}, \label{eq:bnd1}
\end{align}
where we used that $\dsfA^{-\frac12}=0$ and $\rho \tau\le 1$ (see~\eqref{eq:ass_rho_dt}). 
Similarly, we have
\begin{subequations} \begin{align}
\int_0^{+\infty} \|\dddot{w}_{h\tau}(t)\|_\Omega^2 e^{-2\rho t}dt &\lesssim 
\sum_{n\in\polN} \tau \|\dsfA\nmz\|_\Omega^2 e^{-2\rho t^n}, \label{eq:bn2a}\\
\int_0^{+\infty} \|\nabla \deltaht(t)\|_\Omega^2 e^{-2\rho t}dt &\lesssim 
\tau^4 \sum_{n\in\polN} \tau \big( \|\nabla \sfA\nmz\|_\Omega^2 + \tau^2
\|\nabla \dsfA\nmz\|_\Omega^2\big) e^{-2\rho t^n}, \label{eq:bn2b}\\
\int_0^{+\infty} \|\nabla \ddeltaht(t)\|_\Omega^2 e^{-2\rho t}dt &\lesssim 
\tau^4 \sum_{n\in\polN} \tau \|\nabla \dsfA\nmz\|_\Omega^2 e^{-2\rho t^n}. \label{eq:bn2c}
\end{align} \end{subequations}
We can now invoke Corollary~\ref{cor:estim_sfA} (see 
Section~\ref{section_stability_leapfrog}) to conclude the proof.
Indeed, \eqref{eq:bnd_dot_eta} and \eqref{eq:bnd_dddwht}
readily follow from \eqref{eq:bnd_sfA},
\eqref{eq:bnd1}, and \eqref{eq:bn2a}. The estimate~\eqref{eq:bnd_sfA} can also 
be used to bound the first term on the right-hand side of~\eqref{eq:bn2b}. 
To bound the second term on the right-hand side of~\eqref{eq:bn2b} as well as
the right-hand side of~\eqref{eq:bn2c}, we invoke~\eqref{eq:bnd_sfB}, 
recalling that $\sfB^n := \frac{1}{\tau^3}(\sfUnp-3\sfUn+3\sfUnm-\sfUnmm)$,
$\sfB\nmz := \frac12(\sfB^{n+1}+\sfB^n)$, 
$\dsfB\nmz := \frac{1}{\tau}(\sfB^{n+1}-\sfB^n)$, so that $\sfB^{n+1}=\dsfA\nmz$
(notice that $\sfB^n$ is meant to approximate the third-order time-derivative).
To conclude, we observe that
\[
\|\nabla \dsfA\nmz\|_\Omega = \|\nabla \sfB^{n+1}\|_\Omega
\le \|\nabla \sfB\nmz\|_\Omega + \frac12 \tau \|\nabla\dsfB\nmz\|_\Omega
\le \|\nabla \sfB\nmz\|_\Omega + \frac12 (1-\mu_0)^{\frac12} \|\dsfB\nmz\|_\Omega,
\]
where we used~\eqref{eq:mht_L2} owing to the CFL condition. 
\end{proof}

\section{Asymptotically constant-free error upper bound}
\label{section_upper_bound}

Our goal is to bound the error $e=u-\wht$ (see~\eqref{eq:def_e}),
where $u$ is the exact solution and $\wht$ is the $C^2$-reconstruction
of the fully discrete solution defined in~\eqref{eq:defs_uw}.
The evolutionary PDE governing the error can be written as 
\begin{equation} \label{eq:PDE_error}
(\ddot e(t),v)_\Omega + (\nabla e(t),\nabla v)_\Omega = (\eta_f(t),v)_\Omega
+ \langle \calR(t),v\rangle \quad
\forall t\in J,\; \forall v\in V,
\end{equation}
where we introduced the data time-oscillation term $\eta_f\in C^0(J;L^2(\Omega))$ such that
\begin{equation} \label{eq:def:oscill}
\eta_f(t) := f(t)-f_\tau(t),
\end{equation}
and the residual $\calR(t) \in V'$ associated with~\eqref{eq:time_leapfrog} such that
\begin{equation} \label{eq:residual}
\langle \calR(t),v\rangle := (f_\tau(t),v)_\Omega - (\ddwht(t),v)_\Omega - (\nabla \wht(t),\nabla v)_\Omega \quad \forall v\in V,
\end{equation}
with the brackets denoting the duality pairing between $V'$ and $V$. The $\|{\cdot}\|_{V'}$-norm is defined by equipping $V$ with the $H^1$-seminorm, i.e., we set $\|\calR(t)\|_{V'} := \sup_{v\in V,\; \|\nabla v\|_\Omega=1} |\langle \calR(t),v\rangle|$.
The key consistency property we shall use for the residual is the following perturbed Galerkin orthogonality:
\begin{equation} \label{eq:Galerkin}
\langle \calR(t),v_h\rangle = (\nabla\deltaht(t),\nabla v_h)_\Omega \quad \forall v_h\in V_h,
\end{equation}
recalling from~\eqref{eq:def_deltaht} the time-reconstruction error defined as
$\deltaht(t) := \uht(t)-\wht(t)$. \rev{In the proofs below, it is useful to consider} the 
modified residual $\calR^\dagger(t) \in V'$ such that $\calR^\dagger(t):=\calR(t)+ \rev{\Delta} \deltaht(t)$, i.e.,
\begin{equation} \label{eq:residual_dag}
\langle \calRdag(t),v\rangle := (f_\tau(t),v)_\Omega - (\ddwht(t),v)_\Omega - (\nabla \uht(t),\nabla v)_\Omega \quad \forall v\in V,
\end{equation}
which satisfies the exact Galerkin orthogonality
\begin{equation} \label{eq:Galerkin_dag}
\langle \calRdag(t),v_h\rangle = 0 \quad \forall v_h\in V_h.
\end{equation}

In the frequency domain, using obvious notation, 
the error equation~\eqref{eq:PDE_error} becomes
\begin{equation} \label{eq:PDE_error_s}
b_s(\he(s),\hv) = (\hetaf(s),\hv)_\Omega + \langle  \hR(s),\hv\rangle \quad
\forall s\in\polC,\; \forall \hv\in \hV.
\end{equation}
We now derive two bounds on $\tnorm{\he(s)}$, respectively called low-frequency
and high-frequency bounds because in our final error estimate, the first bound will be used for $|s|\le \omega$ and the second bound for $|s|\ge\omega$, where $\omega>0$ is a cutoff frequency. 

\begin{lemma}[Low-frequency bound] \label{lem:low}
The following holds for all $s\in\polC$:
\begin{equation}
\tnorm{\he(s)}^2 \le \rev{\big( 1+40\gamma_s(h)^2\big)}  \|\hR(s)\|_{V'}^2 
+ \frac{\rev{18}}{\rho^2} \|\nabla \hddeltaht(s)\|_\Omega^2
+ \frac{9}{\rho^2} \|\hetaf(s)\|_\Omega^2,
\end{equation}
with the approximation factor $\gamma_s(h)$ defined in~\eqref{eq:def_gamma_s}.
\end{lemma}

\begin{proof}
We need to bound $|s|\|\he(s)\|_\Omega$ and $\|\nabla \he(s)\|_\Omega$ for all $s\in\polC$.

(1) Bound on $|s|\|\he(s)\|_\Omega$.
Let $\hchi_e(s) \in \hV$ solve the adjoint problem
$b_s(\hw,\hchi_e(s)) = |s|^2(\hw,\he(s))_\Omega$ for all $\hw\in \hV$. 
The stability property \eqref{eq:stability} readily implies that
\[
\rho\tnorm{\hchi_e(s)}^2 = \Re\big( b_s(\hchi_e(s),s\hchi_e(s))\big) 
= \Re\big(\bar{s}|s|^2(\hchi_e(s),\he(s))_\Omega\big) \le |s|^2 \tnorm{\hchi_e(s)} \|\he(s)\|_\Omega,
\] 
so that
\[
\tnorm{\hchi_e(s)} \le \frac{|s|}{\rho} |s| \|\he(s)\|_\Omega.
\]
Moreover, testing \eqref{eq:PDE_error_s} with $\hv:=\hchi_e(s)$ gives
\begin{align*}
|s|^2\|\he(s)\|_\Omega^2 &= b_s(\he(s),\hchi_e(s)) = 
(\hetaf(s),\hchi_e(s))_\Omega + \langle  \hR(s),\hchi_e(s)\rangle\\
&= (\hetaf(s),\hchi_e(s))_\Omega + (\nabla \hdeltaht(s),\nabla \hchi_e(s))_\Omega 
+ \langle  \hRdag(s),\hchi_e(s)-\hv_h(s)\rangle,
\end{align*}
where we introduced the Laplace-transformed 
modified residual, $\hRdag(s)$, and we
exploited the exact Galerkin orthogonality property \eqref{eq:Galerkin_dag}
to introduce in the rightmost term an arbitrary discrete function $\hv_h(s)\in \hV_h$
for all $s\in\polC$. Owing to the Cauchy--Schwarz inequality and
the above bound on $\tnorm{\hchi_e(s)}$, we infer that
\begin{align*}
\big| (\hetaf(s),\hchi_e(s))_\Omega + (\nabla \hdeltaht(s),\nabla \hchi_e(s))_\Omega\big| & \le \big( \|\hetaf(s)\|_\Omega^2 + \|\nabla \hddeltaht(s)\|_\Omega^2\big)^{\frac12} \frac{1}{|s|} \tnorm{\hchi_e(s)} \\
& \le \big( \|\hetaf(s)\|_\Omega^2 + \|\nabla \hddeltaht(s)\|_\Omega^2\big)^{\frac12} \frac{1}{\rho} |s| \|\he(s)\|_\Omega.
\end{align*}
Here, we used that $\hddeltaht(s)=s\hdeltaht(s)$. Moreover,
recalling the approximation factor $\gamma_s(h)$ defined
in~\eqref{eq:def_gamma_s} and since $\hv_h(s)$ is arbitrary
in $\hV_h$, we have
\begin{align*}
\inf_{\hv_h(s)\in \hV_h} \big| \langle  \hRdag(s),\hchi_e(s)-\hv_h(s)\rangle \big| 
& \le \gamma_s(h) \|\hRdag(s)\|_{V'} |s| \|\he(s)\|_\Omega.
\end{align*}
Putting the above bounds together gives
\[
|s|\|\he(s)\|_\Omega \le \frac{1}{\rho} \big( \|\hetaf(s)\|_\Omega^2 + \|\nabla \hddeltaht(s)\|_\Omega^2\big)^{\frac12} + \gamma_s(h) \|\hRdag(s)\|_{V'}.
\]
\rev{Since $\|\hRdag(s)\|_{V'} \le \|\hR(s)\|_{V'} + \|\nabla\hdeltaht(s)\|_\Omega$ by the triangle
inequality and since $\gamma_s(h) \le \frac{|s|}{\rho}$ owing to Lemma~\ref{lem:bnd_gamma},
we infer that
\[
|s|\|\he(s)\|_\Omega \le \frac{1}{\rho} \big( \|\hetaf(s)\|_\Omega^2 + \|\nabla \hddeltaht(s)\|_\Omega^2\big)^{\frac12} + \gamma_s(h) \|\hR(s)\|_{V'} + \frac{1}{\rho}\|\nabla \hddeltaht(s)\|_\Omega^2 =: \Upsilon.
\]}

(2) Bound on $\|\nabla \he(s)\|_\Omega$.
Since $\|\nabla \hv\|_\Omega^2 = b_s(\hv,\hv)-s^2\|\hv\|_\Omega^2$ for all $\hv\in \hV$, we
have
\[
\|\nabla \he(s)\|_\Omega^2 \le |b_s(\he(s),\he(s))| + \Upsilon^2,
\]
and it remains to bound the first term on the right-hand side.
Testing \eqref{eq:PDE_error_s} with $\hv:=\he(s)$ gives
\[
b_s(\he(s),\he(s)) = (\hetaf(s),\he(s))_\Omega + \langle  \hR(s),\he(s)\rangle.
\]
Using the Cauchy--Schwarz inequality and $\rho\|\he(s)\|_\Omega
\le |s|\|\he(s)\|_\Omega \le \Upsilon$, this implies that
\[
|b_s(\he(s),\he(s))| \le \|\hR(s)\|_{V'} \|\nabla \he(s)\|_\Omega + \frac{1}{\rho}\|\hetaf(s)\|_\Omega \Upsilon.
\] 
Putting the above bounds together gives
\[
\|\nabla \he(s)\|_\Omega^2 \le \|\hR(s)\|_{V'} \|\nabla \he(s)\|_\Omega + \frac{1}{\rho}\|\hetaf(s)\|_\Omega \Upsilon + \Upsilon^2.
\]
Using Young's inequality and re-arranging the terms, we conclude that
\[
\|\nabla \he(s)\|_\Omega^2 \le \|\hR(s)\|_{V'}^2 + \frac{1}{\rho^2}\|\hetaf(s)\|_\Omega^2 + 3\Upsilon^2.
\]

(3) Bound on $\tnorm{\he(s)}^2$. Combining the bounds from Steps (1) and (2) yields
\[
\tnorm{\he(s)}^2 \le \|\hR(s)\|_{V'}^2 + \frac{1}{\rho^2}\|\hetaf(s)\|_\Omega^2 + 4\Upsilon^2.
\]
The claim follows from $\Upsilon^2\le \frac{2}{\rho^2}\|\hetaf(s)\|_\Omega^2 + \frac{2}{\rho^2}\|\nabla \hddeltaht(s)\|_\Omega^2+ 2\rev{\big(5\gamma_s(h)^2 \|\hR(s)\|_{V'}^2 + \frac54\frac{1}{\rho^2}\|\nabla \hddeltaht(s)\|_\Omega^2\big)}$ \rev{and re-arranging the terms}.
\end{proof}

\begin{lemma}[High-frequency bound] \label{lem:high}
The following holds for all $s\in\polC$ and for all 
$r\ge0$ such that $f\in C^r_{\rm b}(J;L^2(\Omega))$:
\begin{equation}
\tnorm{\he(s)} \le \frac{1}{\rho} \big( \|\nabla \hddeltaht(s)\|_\Omega^2 + \|\hetaf(s)\|_\Omega^2\big)^{\frac12} + \frac{2}{\rho} \frac{1}{|s|^r} \|\hfr(s)\|_\Omega.
\end{equation}
\end{lemma}

\begin{proof}
The triangle inequality gives
\[
\tnorm{\he(s)} \le \tnorm{\hu(s)} + \tnorm{\hwht(s)},
\]
and we bound the two terms on the right-hand side. 

(1) Owing to~\eqref{eq:key_stab_freq}, we infer that
\[
\tnorm{\hu(s)} \le \frac{1}{\rho} \|\hf(s)\|_\Omega = \frac{1}{\rho} \frac{1}{|s|^r} \|\hfr(s)\|_\Omega,
\]
where the last equality follows by invoking the smoothness of the source term $f$ in time. 

(2) Taking the Laplace transform of \eqref{eq:time_leapfrog} and re-organizing the terms gives
\[
b_s(\hwht(s),\hv) = (\hf(s),\hv)_\Omega - (\hetaf(s),\hv)_\Omega - (\nabla \hdeltaht(s),\nabla \hv)_\Omega\quad \forall s\in\polC,\; \forall \hv\in \hV.
\]
Owing to the stability property \eqref{eq:stability}, we infer that
\begin{align}
\rho \tnorm{\hwht(s)}^2 &= \Re\big((\hf(s),s\hwht(s))_\Omega - (\hetaf(s),s\hwht(s))_\Omega - (\nabla \hdeltaht(s),s\nabla \hwht(s))_\Omega \big)\nonumber \\
&\le \Big\{ \big(\|\hf(s)\|_\Omega + \|\hetaf(s)\|_\Omega\big)^2
+ \|\nabla \hddeltaht(s)\|_\Omega^2  \Big\}^{\frac12} \tnorm{\hwht(s)}. \label{eq:bnd_tnorm_w2}
\end{align}
Bounding $\|\hf(s)\|_\Omega$ as in Step~(1) and using that $((a+b)^2+c^2)^{\frac12}\le a+(b^2+c^2)^{\frac12}$ for nonnegative real numbers $a,b,c$, we infer that
\[
\tnorm{\hwht(s)} \le \frac{1}{\rho} \frac{1}{|s|^r} \|\hfr(s)\|_\Omega
+ \frac{1}{\rho} \big( \|\hetaf(s)\|_\Omega^2 + \|\nabla \hddeltaht(s)\|_\Omega^2\big)^{\frac12}.
\]
Putting the above two bounds together proves the claim.
\end{proof}

\begin{theorem}[Error upper bound] \label{thm:upper}
Let the error $e\in C^2(J;V)$ be defined in~\eqref{eq:def_e}
and recall the definition~\eqref{eq:def_calE_rho} of the damped energy norm.
Let the functions $\eta_f(t)\in L^2(\Omega)$ and $\deltaht(t)\in V_h$ 
be defined in~\eqref{eq:def:oscill} and \eqref{eq:def_deltaht}, respectively,
and let the \rev{residual
$\calR(t)\in V'$ be defined in~\eqref{eq:residual}}.  
Let the cutoff frequency $\omega>0$ and the parameter $\rho>0$ be fixed,
and let the approximation factor $\gamma_{\rho,\omega}(h)$ be defined 
in~\eqref{eq:def_gamma_rho_omega}.
Let $r\ge0$ be such that $f \in C^r_{\rm b}(J;L^2(\Omega))$.
The following holds:
\begin{align}
\calE_\rho^2(e)
\le {}& \int_0^{+\infty} \Bigg\{\|\calR(t)\|_{V'}^2 + \rev{\frac{20}{\rho^2} \|\nabla \ddeltaht(t)\|_\Omega^2} \nonumber \\
& + \rev{40\gamma_{\rho,\omega}(h)^2\|\rev{\calR}(t)\|_{V'}^2}
+ \frac{11}{\rho^2} \|\eta_f(t)\|_\Omega^2 
+ \frac{8}{\rho^2} \frac{1}{\omega^{2r}} \|f^{(r)}(t)\|_\Omega^2\Bigg\} e^{-2\rho t} dt. \label{eq:error_upper_bnd}
\end{align}
\end{theorem}

\begin{proof}
Owing to the identity~\eqref{eq:identity}, we infer that
\[
\calE_\rho^2(e) = \int_0^{+\infty} \big\{ \|\dot e(t)\|_\Omega^2 + \|\nabla e(t)\|_\Omega^2 \big\} e^{-2\rho t}dt
= \int_{\rho-\sfi\infty}^{\rho+\sfi\infty} \tnorm{\he(s)}_\Omega^2 ds.
\]
We split the integral on the right-hand side depending on whether
$|s|\le \omega$ or $|s|\ge \omega$. Owing to Lemma~\ref{lem:low}
and the definition of $\gamma_{\rho,\omega}(h)$, we have
\begin{align*}
\int_{\rho-\sfi\infty}^{\rho+\sfi\infty} \tnorm{\he(s)}_\Omega^2 1_{\{|s|\le \omega\}} ds
\le {}& \int_{\rho-\sfi\infty}^{\rho+\sfi\infty} \Bigg\{ \rev{\big(1+40\gamma_{\rho,\omega}(h)^2\big)}
\|\hR(s)\|_{V'}^2  \\
& + \frac{\rev{18}}{\rho^2} \|\nabla \hddeltaht(s)\|_\Omega^2 + \frac{9}{\rho^2} \|\hetaf(s)\|_\Omega^2 \Bigg\} ds,
\end{align*}
where $1_A$ denotes the characteristic function of the subset $A\subset \polC$.
Moreover, invoking Lemma~\ref{lem:high}, we infer that
\begin{align*}
\int_{\rho-\sfi\infty}^{\rho+\sfi\infty} \tnorm{\he(s)}_\Omega^2 1_{\{|s|\ge \omega\}} ds
& \le \int_{\rho-\sfi\infty}^{\rho+\sfi\infty} \Big\{ \frac{2}{\rho^2} \|\nabla \hddeltaht(s)\|_\Omega^2 + \frac{2}{\rho^2} \|\hetaf(s)\|_\Omega^2 + \frac{8}{\rho^2} \frac{1}{\omega^{2r}} \|\hfr(s)\|_\Omega^2 \Big\} ds.
\end{align*}
Putting the above two bounds together proves the assertion. 
\end{proof}

\begin{remark}[Theorem~\ref{thm:upper}]
The estimate~\eqref{eq:error_upper_bnd} bounds the damped energy norm of the error
(recall that the damping parameter $\rho$ is typically proportional to the reciprocal
of the simulation time $T_*$) in terms of the dual norm of the residual, $\calR$
(representative of the space discretization error), the time-reconstruction error, 
$\deltaht$ (representative of the time discretization error), the data time-oscillation term,
$\eta_f$, and a term depending on higher-order time-derivatives of $f$.
\end{remark}

\begin{remark}[\rev{Higher-order terms}] \label{rem:hot_up}
\rev{The error upper bound~\eqref{eq:error_upper_bnd} fits the form~\eqref{eq:outline_thm12} if one sets
\[
\textup{h.o.t.} := \int_0^{+\infty} \Bigg\{\rev{40}\gamma_{\rho,\omega}(h)^2\|\rev{\calR}(t)\|_{V'}^2
+ \frac{11}{\rho^2} \|\eta_f(t)\|_\Omega^2 
+ \frac{8}{\rho^2} \frac{1}{\omega^{2r}} \|f^{(r)}(t)\|_\Omega^2\Bigg\} e^{-2\rho t} dt.
\]
Let us motivate that the three terms on the right-hand side can be considered as higher-order terms.
First, we observe
that the data time-oscillation term $\eta_f$ converges to third-order in $\tau$ owing to
Lemma~\ref{lem:etaf} provided $f \in C^3_{\rm b}(J;L^2(\Omega))$. Concerning the other two
terms on the right-hand side, we adapt the arguments of \cite[Corollary~5.2]{Theo:23}.
We first set the cutoff frequency as
\begin{equation*}
\omega^2 = \left (
\left (\frac{\ell_\Omega}{h}\right )^{\frac{\theta}{2}}(\rho\ell_\Omega)^\beta -1
\right ) \rho^2,
\end{equation*}
where the exponent $\beta\ge0$ will be chosen later on (the above right-hand side is positive if $h$ is small enough).
Owing to the bound~\eqref{eq:bnd_gamma} on $\gamma_{s}(h)$, we obtain
\begin{equation} \label{eq:bnd_rho1}
\gamma_{\rho,\omega}(h)
\leq 2C_{\mathrm{app}} C_{\mathrm{ell}} (\rho\ell_\Omega)^{1-\beta} \left (\frac{h}{\ell_\Omega}\right )^{\frac{\theta}{2}}.
\end{equation}
Assuming (to fix the ideas) that $h\leq \ell_\Omega/16$,
$\rho\ell_\Omega\le 1$, and since $\theta>\frac12$, we infer that 
$\omega^2 \geq \frac12 \rho^2 \left (\frac{\ell_\Omega}{h}\right )^{\frac{\theta}{2}} (\rho\ell_\Omega)^{-\beta}$. 
Observing that 
\begin{equation} \label{eq:bnd_rho2}
\frac{8}{\rho^2} \frac{1}{\omega^{2r}} \|f^{(r)}(t)\|_\Omega^2
\leq
\frac{2^{3+r}}{\rho^{2r+2}} (\rho \ell_\Omega)^{\beta r} \left (\frac{h}{\ell_\Omega}\right )^{\frac{r\theta}{2}} \|f^{(r)}(t)\|_\Omega^2,
\end{equation}
we select $r$ so that $r\ge \frac{2k+1}{\theta} \ge 4k+2$. Finally, we can choose
$\beta = \frac{3+2r}{1+r}\approx 2+\frac{1}{r}\approx 2$ to balance the powers of $\rho$ in the bounds~\eqref{eq:bnd_rho1}-\eqref{eq:bnd_rho2}.}
\end{remark}

\section{Error lower bound}
\label{section_lower_bound}

\rev{In this section, we establish an error lower bound. We need to consider only the first
two terms on the right-hand side of~\eqref{eq:error_upper_bnd} since the other three terms
can be considered as higher-order terms (see Remark~\ref{rem:hot_up}). However, the second
term on the right-hand side, involving $\nabla \ddeltaht$, is, at the same time, an
error indicator and an error. Therefore, we focus here on bounding the first term on the 
right-hand side of~\eqref{eq:error_upper_bnd} involving the dual norm of the residual.

We consider the $L^2$-orthogonal projection
$\pi_h: L^2(\Omega) \to V_h$ defined by $(\pi_h(\phi),v_h)_\Omega = (\phi,v_h)_\Omega$
for all $\phi \in L^2(\Omega)$ and all $v_h \in V_h$. Under mild assumptions
on the mesh grading, which accommodate newest vertex bisection in two dimensions (see
\cite{BanYs:14,GasHS:16} and also \cite[Remark 22.23]{Ern_Guermond_FEs_I_2021}),
we have, for all $\phi \in V$,
\begin{equation}
\label{eq:pih}
\|\phi-\pi_h (\phi)\|_\Omega
\lesssim
h \|\nabla \phi\|_\Omega,
\qquad
\|\nabla \pi_h (\phi)\|_\Omega
\lesssim
\|\nabla \phi\|_\Omega.
\end{equation}

\begin{theorem}[Error lower bound] \label{thm:lower}
Assume $f\in C^{r}_{\rm b}(J;L^2(\Omega))$.
Assume the CFL condition~\eqref{eq:CFL}.
Assume that the mesh is graded so that \eqref{eq:pih} is satisfied. 
The following holds for all $r\ge2$:
\begin{align}
\int_0^{+\infty} \|\calR(t)\|_{V'}^2 e^{-2\rho t}dt \lesssim {}&
\int_0^{+\infty} \Big\{\|\dot e(t)\|_\Omega^2+\|\nabla e(t)\|_\Omega^2\Big\} e^{-2\rho t}dt
+
\int_0^{+\infty} \|\nabla \deltaht(t)\|_\Omega^2 e^{-2\rho t}dt\nonumber \\
& + h^2\int_0^{+\infty} \|\eta_f(t)\|_\Omega^2  e^{-2\rho t}dt
+ \frac{h^{2r}}{\rho^2}\int_0^{+\infty} \|f^{(r)}(t)\|_\Omega^2 e^{-2\rho t} dt.
\label{eq:lower}
\end{align}
\end{theorem}}

\rev{\begin{proof}
(1) The triangle inequality gives $\|\calR(t)\|_{V'}\le \|\calRdag(t)\|_{V'}
+ \|\nabla\deltaht(t)\|_\Omega$, where $\calRdag(t)\in V'$ is the modified residual
defined in~\eqref{eq:residual_dag}. 
To bound the norm of $\calRdag(t)$ in $V'$, we pick an arbitrary
$v \in V$ with $\|\nabla v\|_\Omega = 1$. Invoking the Galerkin
orthogonality property~\eqref{eq:Galerkin_dag} gives
\begin{align*}
\langle \calRdag(t),v \rangle
={}&
\langle \calRdag(t),(I-\pi_h)(v) \rangle
\\
={}&
(\ddot e(t),(I-\pi_h)(v))_\Omega + (\nabla e(t),\nabla (I-\pi_h)(v))_\Omega,
+
(\nabla \deltaht(t),\nabla(I-\pi_h)(v))_\Omega - (\eta_f(t),(I-\pi_h)(v))_\Omega
\\
={}&
((I-\pi_h)\ddot e(t),(I-\pi_h)(v))_\Omega + (\nabla e(t),\nabla (I-\pi_h)(v))_\Omega,
+
(\nabla \deltaht(t),\nabla(I-\pi_h)(v))_\Omega \\
&- (\eta_f(t),(I-\pi_h)(v))_\Omega,
\end{align*}
where we employed the $L^2$-orthogonality property of $\pi_h$. Invoking
the properties of $\pi_h$ in~\eqref{eq:pih}, we infer that
\begin{equation*} 
\|\calRdag(t)\|_{V'} \lesssim h\|(I-\pi_h)\ddot e(t)\|_\Omega 
+ \|\nabla e(t)\|_\Omega + \|\nabla \deltaht(t)\|_\Omega + h\|\eta_f(t)\|_\Omega,
\end{equation*}
for all $t \geq 0$. Owing to the triangle inequality, we infer that
\begin{equation} \label{eq:bnd_calR}
\|\calR(t)\|_{V'} \lesssim h\|(I-\pi_h)\ddot e(t)\|_\Omega 
+ \|\nabla e(t)\|_\Omega + \|\nabla \deltaht(t)\|_\Omega + h\|\eta_f(t)\|_\Omega,
\end{equation}
and it remains to bound the first term on the right-hand side.

(2) Owing to the identity~\eqref{eq:identity}, we infer that
\[
h^2\int_0^{+\infty} \|(I-\pi_h)(\ddot e(t))\|_\Omega^2 e^{-2\rho t}dt
= h^2 \int_{\rho-\sfi\infty}^{\rho+\sfi\infty} |s|^2 \|(I-\pi_h)\hde(s)\|_\Omega^2 ds.
\]
We split the integral on the right-hand side depending on whether
$|s|\le h^{-1}$ or $|s|\ge h^{-1}$. On the one hand, we have
\[
h^2 \int_{\rho-\sfi\infty}^{\rho+\sfi\infty} |s|^2 \|(I-\pi_h)\hde(s)\|_\Omega^2 
1_{|s|\le h^{-1}}ds \le \int_{\rho-\sfi\infty}^{\rho+\sfi\infty} \|(I-\pi_h)\hde(s)\|_\Omega^2ds
= \int_0^{+\infty} \|(I-\pi_h)\dot e(t)\|_\Omega^2 e^{-2\rho t}dt. 
\]
On the other hand, we have 
\begin{align*}
h^2 \int_{\rho-\sfi\infty}^{\rho+\sfi\infty} |s|^2 \|(I-\pi_h)\hde(s)\|_\Omega^2 1_{|s|> h^{-1}}ds 
& \le h^2 \int_{\rho-\sfi\infty}^{\rho+\sfi\infty} |s|^2 \|\hat{\dot u}(s)\|_\Omega^2 1_{|s|> h^{-1}}ds  \\
& \le \frac{h^{2r}}{\rho^2} \int_0^{+\infty} \|f^{(r)}(t)\|_\Omega^2 e^{-2\rho t}dt. 
\end{align*}
Combining these bounds proves that
\begin{equation} \label{eq:bnd_ddot_e}
h^2\int_0^{+\infty} \|(I-\pi_h)(\ddot e(t))\|_\Omega^2 e^{-2\rho t}dt \le
\int_0^{+\infty} \|(I-\pi_h)\dot e(t)\|_\Omega^2 e^{-2\rho t}dt + 
\frac{h^{2r}}{\rho^2} \int_0^{+\infty} \|f^{(r)}(t)\|_\Omega^2 e^{-2\rho t}dt.
\end{equation}
Since $\|(I-\pi_h)\dot e(t)\|_\Omega \le \|\dot e(t)\|_\Omega$, 
the combination of~\eqref{eq:bnd_calR} and~\eqref{eq:bnd_ddot_e} concludes the proof.
\end{proof}}

\begin{remark}[\rev{Higher-order terms}] \label{rem:hot_low}
\rev{The error lower bound~\eqref{eq:lower}
fits the form~\eqref{eq:low_intro} if one sets
\[
\textup{h.o.t.} := h^2 \int_0^{+\infty} \|\eta_f(t)\|_\Omega^2  e^{-2\rho t}dt
+ \frac{h^{2r}}{\rho^2} \int_0^{+\infty} \|f^{(r)}(t)\|_\Omega^2  e^{-2\rho t} dt.
\]
We refer the reader to Remark~\ref{rem:hot_up} for the discussion on the data time-oscillation term $\eta_f$. Moreover, we can take here $r\ge k+1$.}
\end{remark}

\begin{remark}[\rev{Alternative error lower bound}]
\rev{It is also possible to establish an error lower bound without invoking
the $H^1$-stability of the $L^2$-orthogonal projection, but this leads
to the additional term $\frac{h^4}{\rho^2}\int_0^{+\infty} \|\ddot f(t)\|_\Omega^2 e^{-2\rho t}dt$
on the right-hand side of~\eqref{eq:lower}.
This term is not fully satisfactory as it is not of higher-order.
% it decreases to second-order with the mesh size and not the time step.
To establish the claim, the proof
proceeds again in two steps. The first step is similar to the above one (but invokes
any $H^1$-stable quasi-interpolation operator instead of the $L^2$-orthogonal projection), 
leading to (compare with~\eqref{eq:bnd_calR})
\[
\|\calR(t)\|_{V'} \lesssim h\|\ddot e(t)\|_\Omega 
+ \|\nabla e(t)\|_\Omega + \|\nabla \deltaht(t)\|_\Omega + h\|\eta_f(t)\|_\Omega.
\]
To bound $h\|\ddot e(t)\|_\Omega$, we consider as above the Laplace transform.
In the low-frequency regime ($|s|\le h^{-1}$), we have
\[
h\|\hat{\ddot e}(s)\|_\Omega = h |s| \|\hat{\dot e}(s)\|_\Omega \le \|\hat{\dot e}(s)\|_\Omega.
\]
In the high-frequency regime ($|s|\ge h^{-1}$), we first invoke the triangle inequality so that
$h\|\hat{\ddot e}(s)\|_\Omega \le h |s|^2 \|\hat u(s)\|_\Omega + h |s|^2 \|\hwht(s)\|_\Omega$.
The first term on the right-hand side is bounded as
\[
h|s|^2 \|\hu(s)\|_\Omega \le h|s| \tnorm{\hu(s)} \le \frac{h}{\rho} |s| \|\hf(s)\|_\Omega \le \frac{h^r}{\rho} \|\hfr(s)\|_\Omega.
\] 
On the other hand, we have $h |s|^2 \|\hwht(s)\|_\Omega \le % h\frac{1}{\omega} 
\|\hat{\dddot w}_{h\tau}(s)\|_\Omega$, and we invoke the bound~\rev{\eqref{eq:bnd_dddwht}} 
from Lemma~\ref{lem:bnd_eta} together with the CFL condition. Altogether, this gives
\[
\int_0^{+\infty} h^2 \|\ddot e(t)\|_\Omega e^{-2\rho t}dt \lesssim 
\int_0^{+\infty} \Big\{\|\dot e(t)\|_\Omega^2 
+ \frac{h^{2r}}{\rho^2} \|f^{(r)}(t)\|_\Omega^2 
+ \frac{h^4}{\rho^2} \|\ddot f(t)\|_\Omega^2 \Big\} e^{-2\rho t}dt,
\]
whence the claim.}
\end{remark}

\section{Numerical results}
\label{section_numerical_results}

In this section, we present numerical results to assess the a posteriori error estimates derived in the previous sections. 

\subsection{Setting}

Here, we describe the discretization parameters, the error measures, and how we estimate the 
space discretization error. In all cases, we consider the one-dimensional domain 
$\Omega := (-L,L)$ with $L := 10$, and set the computational time to $T_\star := 1000$. 
For the values of $\rho$ employed hereafter,
this final time is such that $e^{-\rho T_\star} \leq 5 {\cdot} 10^{-6}$ in all the 
simulations, so that stopping the simulation at $t=T_*$ has little effect.

\subsubsection{Discretization parameters}

We use uniform meshes, and continuous finite elements of degree
$k \in\{1,2,3\}$. We employ mesh sizes ranging from $h=10$ to $h=0.039$, i.e., the interval
$\Omega$ is divided into at least $2$ and at most $512$ elements.
For each $k$, we denote by $\alpha_k$ the largest value for which the leapfrog scheme
is stable with the time step $\tau = \alpha_{k} h$.  
We empirically found these values to be $\alpha_1 = 0.59$,
$\alpha_2 = 0.26$, and $\alpha_3 = 0.15$. In our examples, given $h$ and $k$, 
we fix the time step 
by setting $\tau := r\alpha_k h$ with $r \in (0,1)$. 
Unless explicitly specified, we use $r:=0.9$.

When $k=3$, we also employ finer time-step values so that the order of the 
time discretization error matches that of the space discretization error. 
To do so, we fix $\tau_0 = r_0 \alpha_3 h_0$
on the coarsest mesh (i.e., $h_0 = 10$), and then adjust $\tau$ so that
$\tau^2/h^3$ remains constant for all the finer meshes. 
This is indicated by the notation ``$\tau^2 \sim h^3$'' in the graphs below.

\subsubsection{Error measures}

For convenience, we use the following shorthand notation for the errors, where $z$
stands either for $u$ or for $w$:
\begin{subequations} \begin{alignat}{2}
e^2(z) & := \int_0^{+\infty} \!\! \|u(t)-\zht(t)\|_\Omega^2 e^{-2\rho t} dt,
&\quad
e^2(\sfU) &:= \tau \sum_{n = 0}^{+\infty} \|u(t^n)-\sfU^n\|_\Omega^2 e^{-2\rho t^n}, \\
e_x^2(z) &:= \int_0^{+\infty} \!\!\|\nabla(u(t)-\zht(t))\|_\Omega^2 e^{-2\rho t} dt,
&\quad
e_x^2(\sfU) &:= \tau \sum_{n = 0}^{+\infty} \|\nabla(u(t^n)-\sfU^n)\|_\Omega^2 e^{-2\rho t^n}, \\
e_t^2(z) &:= \int_0^{+\infty} \|\dot u(t)-\dzht(t)\|_\Omega^2 e^{-2\rho t} dt,
&\quad
e_t^2(\sfU) &:= \tau \sum_{n = 0}^{+\infty} \|\dot u(t^n)-\tfrac{\sfU^{n+1}-\sfU^{n-1}}{\tau}\|_\Omega^2 e^{-2\rho t^n}.
\end{alignat} \end{subequations}
For the error measured in the damped energy-norm, we use the shorthand notation
\begin{equation}
\calE_\rho^2 := \calE_\rho^2(u-\wht) = e_t^2(w)+e_x^2(w).
\end{equation}
In practice, all the integrals and sums are computed up to $t=T_\star$. As argued above,
$T_\star$ has been chosen in such a way that the tail of the integrals and 
series is negligible.

\subsubsection{Estimators}

To estimate the space discretization error, we construct a continuous piecewise
polynomial function $\sigma_{h\tau}$ with vanishing spatial mean value such that
\begin{equation}
\partial_x \sigma_{h\tau}(t) = f_\tau(t)-\ddot w_{h\tau}(t) \quad \forall t\in J.
\end{equation}
A simple integration by parts reveals that
\begin{equation}
\|\calR(t)\|_{V'}
\leq
\|\partial_x w_{h\tau}(t)+\sigma_{h\tau}(t)\|_\Omega.
\end{equation}
As a result, introducing the quantities
\begin{subequations} \begin{align}
R^2
&:=
\int_0^{+\infty} \!\!\|\partial_x w_{h\tau}(t)+\sigma_{h\tau}(t)\|^2_\Omega e^{-2\rho t} dt, \\
M^2
&:=
\int_0^{+\infty} \!\! \frac{1}{\rho^2}
\|\partial_x (\dot u_{h\tau}(t) - \dot w_{h\tau}(t))\|_\Omega^2 e^{-2\rho t} dt,
\end{align} \end{subequations}
and using the following shorthand notation for the estimator:
\begin{equation}
\Lambda_\rho^2 := R^2 + \rev{20} M^2,
\end{equation}
the error upper bound \eqref{eq:error_upper_bnd} rewrites
\begin{equation}
\calE_\rho^2 \leq \Lambda_\rho^2 + \operatorname{h.o.t.}
\end{equation}

\subsection{Benchmark solutions}

We consider two different analytical solutions
corresponding to ``standing'' and ``propagating'' waves. 

\subsubsection{Standing wave}
We set
\begin{equation}
f(t,x) := -2(t-t_0)e^{-(t-t_0)^2}\sin\big(a(x-L)\big),
\end{equation}
where $t_0 := 4$ and $a := m\pi/(2L)$ with $m := 5$. Although $f$ does
not vanish at $t=0$, we do have
$|f(0,\cdot)| \leq 8 {\cdot} e^{-16} \leq 0.9 {\cdot} 10^{-6}$,
$|\dot f(0,\cdot)| \leq 62 {\cdot} e^{-16} \leq 7.0 {\cdot} 10^{-6}$
and
$|\ddot f(0,\cdot)| \leq 464 {\cdot} e^{-16} \leq 5.3 {\cdot} 10^{-5}$, whereas
$\max_{t,x} |f(t,x)| \geq 0.5$. The corresponding solution reads
\begin{equation}
u(t,x)
=
\Re \big(\psi(t-t_0)\big) 
\sin(a(x-L)), \qquad
\psi(t)
:= e^{\sfi a t} \int_{-\infty}^{t} e^{-\theta^2-\sfi a \theta} d\theta.
\end{equation}
To see this, observe that $\ddot \psi(t) = (-2t+\sfi a)e^{-t^2} - a^2\psi(t)$
so that
$\Re \big(\ddot \psi(t)+a^2\psi(t)\big) = -2te^{-t^2}.$
We observe that $|u(0,\cdot)| \leq 1.1 {\cdot} 10^{-8}$
and $|\dot u(0,\cdot)| \leq 8.3 {\cdot} 10^{-8}$, whereas
$\max_{t,x} |u(t,x)| \geq 4 {\cdot} 10^{-2}$, so that setting the initial
conditions to zero produces an error that is (much) smaller than the discretization error.
The time profiles of $f$ and $u$ are shown in Figure \ref{figure_solution_standing}.
In practice, we compute $\psi$ with the {\tt erf} function of the {\tt Faddeeva} software
package by observing that
\begin{equation}
\psi(t)
=
e^{-\frac{a^2}{4}+\sfi at} \int_{-\infty}^t e^{-(\theta+\sfi\frac{a}{2})^2} d\theta
=
e^{-\frac{a^2}{4}+\sfi at} \int_{-\infty}^{t+\sfi\frac{a}{2}} e^{-z^2} dz.
\end{equation}

\begin{figure}
\caption{Time profiles in the standing wave example}
\label{figure_solution_standing}
\end{figure}

\subsubsection{Propagating wave} 
We set
\begin{equation}
f(t,x)
:=
-2(t-t_0)e^{-(t-t_0)^2} e^{-x^2},
\end{equation}
where, again, $t_0 := 4$. The solution in full space reads
\begin{equation}
u_\infty(t,x)
:=
\frac{1}{4}
\left (
e^{\frac12(t-t_0+x)^2} \int_{-\infty}^{t-t_0+x} e^{-\frac12\tau^2} d\tau
+
e^{\frac12(t-t_0-x)^2} \int_{-\infty}^{t-t_0-x} e^{-\frac12\tau^2} d\tau
\right )
e^{-(x^2+(t-t_0)^2)}.
\end{equation}
This expression is obtained from the representation
\begin{equation}
u_\infty(t,x)
=
\frac{1}{2}
\int_{y \in \Real} \int_{\theta > |y|} f(t-\theta,x-y) d\theta dy,
\end{equation}
which stems from the Green function $\frac12\mathbf 1_{t > |x|}$
of the one-dimensional wave equation in free space.
Notice that $u_\infty$ does not satisfy the Dirichlet boundary conditions
at $x = \pm L$. However, we can obtain the correct solution by using
the mirror image principle, and setting
\begin{equation}
u(t,x) := u_\infty(t,x) + \sum_{n=1}^{+\infty} (-1)^n \Big\{
u_\infty(t, 2nL+(-1)^n x)
+
u_\infty(t,-2nL+(-1)^n x)
\Big\}.
\end{equation}
One readily checks that $u$ satisfies the boundary conditions at $x \pm L$.
As above, we do not have $u(0,\cdot) = 0$, but the initial values are small
enough so that they do not affect the numerical experiments. Notice that $u$
in fact solves the wave equation over $J \times \mathbb R$ with
right-hand side
\begin{equation}
\widetilde f(t,x) := f(t,x) + \sum_{n=1}^{+\infty} (-1)^n \Big\{
f(t, 2nL+(-1)^n x)
+
f(t,-2nL+(-1)^n x)
\Big\},
\end{equation}
but the difference $\widetilde f - f$ is small enough over $\Omega$ so that it does
not impact the accuracy of our numerical experiments. For any finite time $t \geq 0$ and any
$x \in \Omega$, the terms in the series defining $u(t,x)$ decay exponentially.
For our simulations, we simply compute the sum over $n\in\{1{:}50\}$. As above,
we use the {\tt Faddeeva} software package to evaluate the integral of the Gaussian
numerically. Space profiles of $u$ at different times are shown in Figure \ref{figure_solution_propagating}.

\begin{figure}
\caption{Propagating wave: space profiles of the solution at various times}
\label{figure_solution_propagating}
\end{figure}

\subsection{Accuracy of the reconstructions}

Our first goal is to verify that the reconstructions
$u_{h\tau}$ and $w_{h\tau}$ obtained from the time-step values $(\sfU^n)_{n\in\polN}$
have the expected accuracy. We focus on the standing wave example with $\rho := 0.02$.
As can be seen from Figures \ref{figure_reconstruction}, \ref{figure_reconstruction_space}
and \ref{figure_reconstruction_time}, all the errors are very close,
as anticipated. We also observe the expected convergence rates.
We obtained very similar results with different values of $\rho$ as well as
in the propagating wave example. For the sake of shortness, we do not
reproduce all the curves here.

\begin{figure}
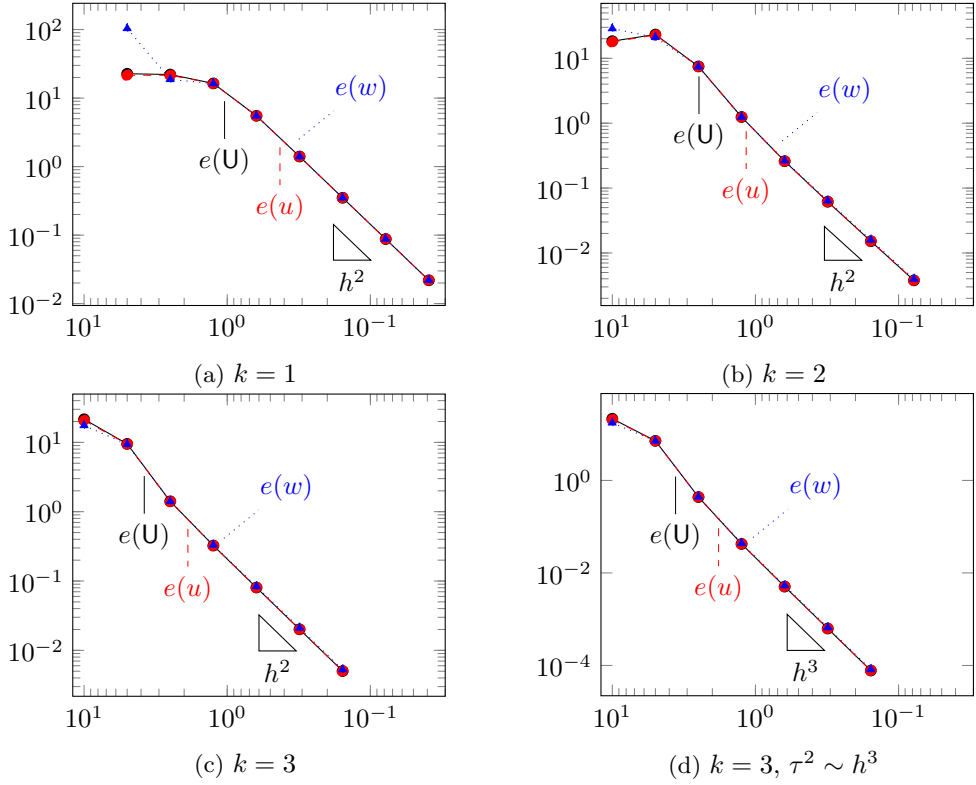

\begin{plotpagetitle}{}{$k=1$}
\ploteUv{figures/standing/data/rho0.02_p0_r0.9/data.txt} \plotpin{0.2}{0.5cm}{-90}{$e(\sfU)$}{black};
\ploteuv{figures/standing/data/rho0.02_p0_r0.9/data.txt} \plotpin{0.4}{0.5cm}{-90}{$e(u)$}{red};
\plotewv{figures/standing/data/rho0.02_p0_r0.9/data.txt} \plotpin{0.5}{0.5cm}{ 45}{$e(w)$}{blue};
\SlopeTriangle{0.7}{-0.1}{0.15}{-2}{$h^2$}{}
\end{plotpagetitle}
\begin{plotpagetitle}{}{$k=2$}
\ploteUv{figures/standing/data/rho0.02_p1_r0.9/data.txt} \plotpin{0.2}{0.5cm}{-90}{$e(\sfU)$}{black};
\ploteuv{figures/standing/data/rho0.02_p1_r0.9/data.txt} \plotpin{0.4}{0.5cm}{-90}{$e(u)$}{red};
\plotewv{figures/standing/data/rho0.02_p1_r0.9/data.txt} \plotpin{0.5}{0.5cm}{ 45}{$e(w)$}{blue};
\SlopeTriangle{0.6}{-0.1}{0.15}{-2}{$h^2$}{}
\end{plotpagetitle}

\begin{plotpagetitle}{}{$k=3$}
\ploteUv{figures/standing/data/rho0.02_p2_r0.9/data.txt} \plotpin{0.2}{0.5cm}{-90}{$e(\sfU)$}{black};
\ploteuv{figures/standing/data/rho0.02_p2_r0.9/data.txt} \plotpin{0.4}{0.5cm}{-90}{$e(u)$}{red};
\plotewv{figures/standing/data/rho0.02_p2_r0.9/data.txt} \plotpin{0.5}{0.5cm}{ 45}{$e(w)$}{blue};
\SlopeTriangle{0.5}{-0.1}{0.15}{-2}{$h^2$}{}
\end{plotpagetitle}
\begin{plotpagetitle}{}{$k=3$, $\tau^2 \sim h^3$}
\ploteUv{figures/standing/data/rho0.02_p2_sts/data.txt} \plotpin{0.2}{0.5cm}{-90}{$e(\sfU)$}{black};
\ploteuv{figures/standing/data/rho0.02_p2_sts/data.txt} \plotpin{0.4}{0.5cm}{-90}{$e(u)$}{red};
\plotewv{figures/standing/data/rho0.02_p2_sts/data.txt} \plotpin{0.5}{0.5cm}{ 45}{$e(w)$}{blue};
\SlopeTriangle{0.5}{-0.1}{0.15}{-3}{$h^3$}{}
\end{plotpagetitle}
\caption{Standing wave: reconstruction errors on the solution}
\label{figure_reconstruction}
\end{figure}

\begin{figure}
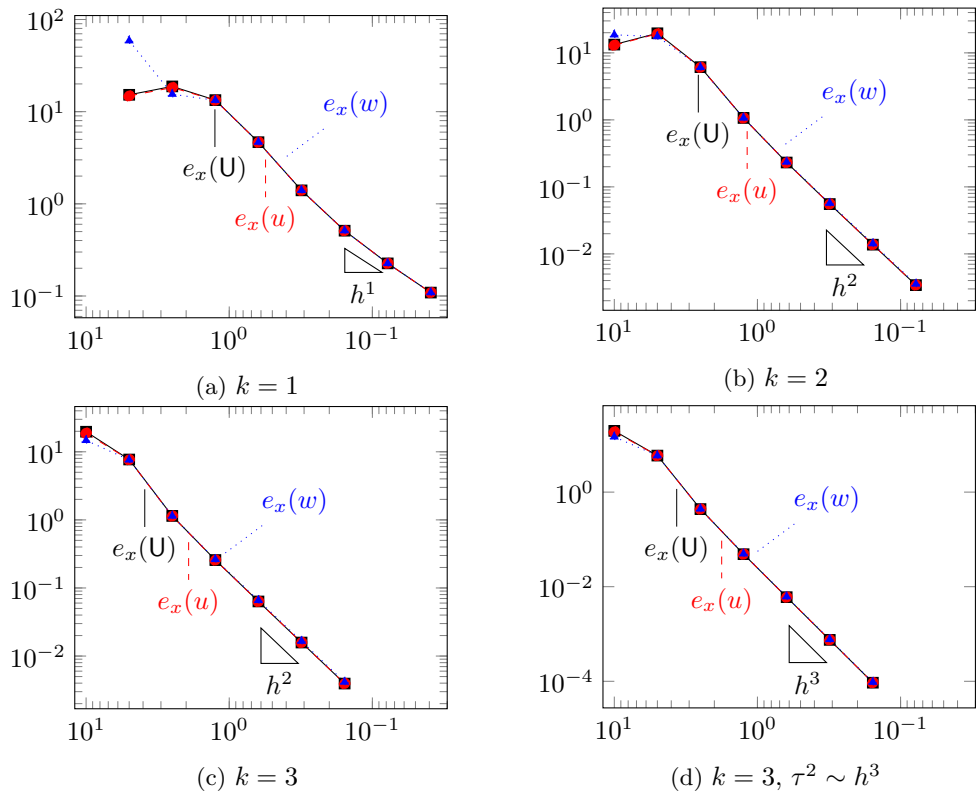

\begin{plotpagetitle}{}{$k=1$}
\ploteUx{figures/standing/data/rho0.02_p0_r0.9/data.txt} \plotpin{0.2}{0.5cm}{-90}{$e_x(\sfU)$}{black};
\ploteux{figures/standing/data/rho0.02_p0_r0.9/data.txt} \plotpin{0.4}{0.5cm}{-90}{$e_x(u)$}{red};
\plotewx{figures/standing/data/rho0.02_p0_r0.9/data.txt} \plotpin{0.5}{0.5cm}{ 45}{$e_x(w)$}{blue};
\SlopeTriangle{0.725}{-0.1}{0.15}{-1}{$h^1$}{}
\end{plotpagetitle}
\begin{plotpagetitle}{}{$k=2$}
\ploteUx{figures/standing/data/rho0.02_p1_r0.9/data.txt} \plotpin{0.2}{0.5cm}{-90}{$e_x(\sfU)$}{black};
\ploteux{figures/standing/data/rho0.02_p1_r0.9/data.txt} \plotpin{0.4}{0.5cm}{-90}{$e_x(u)$}{red};
\plotewx{figures/standing/data/rho0.02_p1_r0.9/data.txt} \plotpin{0.5}{0.5cm}{ 45}{$e_x(w)$}{blue};
\SlopeTriangle{0.6}{-0.1}{0.15}{-2}{$h^2$}{}
\end{plotpagetitle}

\begin{plotpagetitle}{}{$k=3$}
\ploteUx{figures/standing/data/rho0.02_p2_r0.9/data.txt} \plotpin{0.2}{0.5cm}{-90}{$e_x(\sfU)$}{black};
\ploteux{figures/standing/data/rho0.02_p2_r0.9/data.txt} \plotpin{0.4}{0.5cm}{-90}{$e_x(u)$}{red};
\plotewx{figures/standing/data/rho0.02_p2_r0.9/data.txt} \plotpin{0.5}{0.5cm}{ 45}{$e_x(w)$}{blue};
\SlopeTriangle{0.5}{-0.1}{0.15}{-2}{$h^2$}{}
\end{plotpagetitle}
\begin{plotpagetitle}{}{$k=3$, $\tau^2 \sim h^3$}
\ploteUx{figures/standing/data/rho0.02_p2_sts/data.txt} \plotpin{0.2}{0.5cm}{-90}{$e_x(\sfU)$}{black};
\ploteux{figures/standing/data/rho0.02_p2_sts/data.txt} \plotpin{0.4}{0.5cm}{-90}{$e_x(u)$}{red};
\plotewx{figures/standing/data/rho0.02_p2_sts/data.txt} \plotpin{0.5}{0.5cm}{ 45}{$e_x(w)$}{blue};
\SlopeTriangle{0.5}{-0.1}{0.15}{-3}{$h^3$}{}
\end{plotpagetitle}
\caption{Standing wave: reconstruction errors on the space derivative}
\label{figure_reconstruction_space}
\end{figure}

\begin{figure}
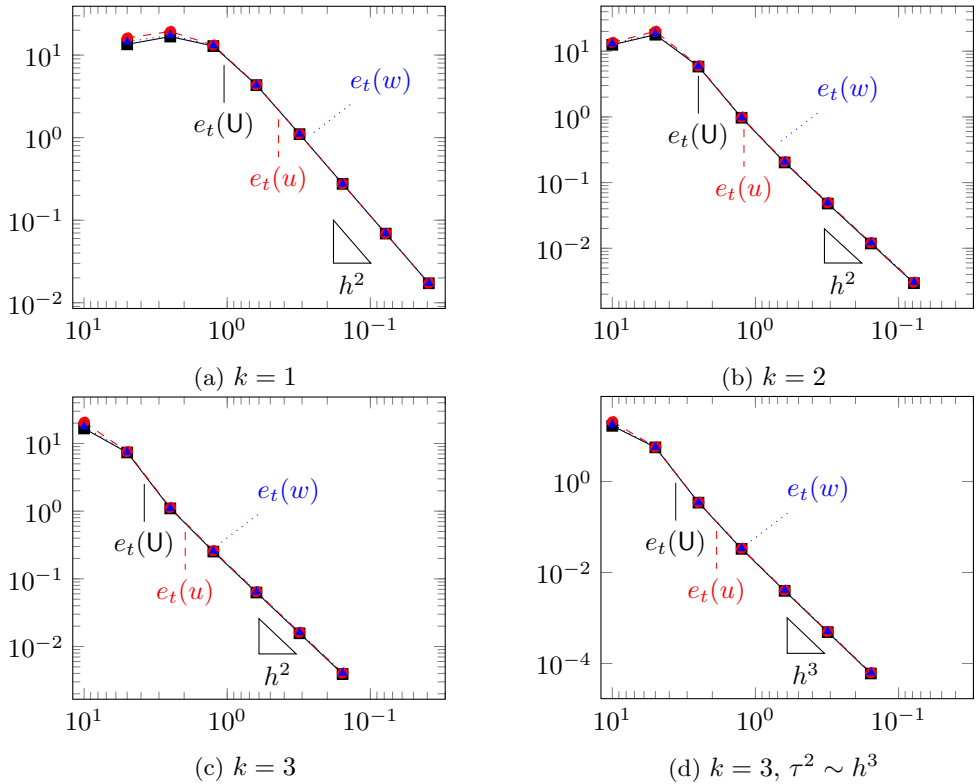

\begin{plotpagetitle}{}{$k=1$}
\ploteUt{figures/standing/data/rho0.02_p0_r0.9/data.txt} \plotpin{0.2}{0.5cm}{-90}{$e_t(\sfU)$}{black};
\ploteut{figures/standing/data/rho0.02_p0_r0.9/data.txt} \plotpin{0.4}{0.5cm}{-90}{$e_t(u)$}{red};
\plotewt{figures/standing/data/rho0.02_p0_r0.9/data.txt} \plotpin{0.5}{0.5cm}{ 45}{$e_t(w)$}{blue};
\SlopeTriangle{0.7}{-0.1}{0.15}{-2}{$h^2$}{}
\end{plotpagetitle}
\begin{plotpagetitle}{}{$k=2$}
\ploteUt{figures/standing/data/rho0.02_p1_r0.9/data.txt} \plotpin{0.2}{0.5cm}{-90}{$e_t(\sfU)$}{black};
\ploteut{figures/standing/data/rho0.02_p1_r0.9/data.txt} \plotpin{0.4}{0.5cm}{-90}{$e_t(u)$}{red};
\plotewt{figures/standing/data/rho0.02_p1_r0.9/data.txt} \plotpin{0.5}{0.5cm}{ 45}{$e_t(w)$}{blue};
\SlopeTriangle{0.6}{-0.1}{0.15}{-2}{$h^2$}{}
\end{plotpagetitle}

\begin{plotpagetitle}{}{$k=3$}
\ploteUt{figures/standing/data/rho0.02_p2_r0.9/data.txt} \plotpin{0.2}{0.5cm}{-90}{$e_t(\sfU)$}{black};
\ploteut{figures/standing/data/rho0.02_p2_r0.9/data.txt} \plotpin{0.4}{0.5cm}{-90}{$e_t(u)$}{red};
\plotewt{figures/standing/data/rho0.02_p2_r0.9/data.txt} \plotpin{0.5}{0.5cm}{ 45}{$e_t(w)$}{blue};
\SlopeTriangle{0.5}{-0.1}{0.15}{-2}{$h^2$}{}
\end{plotpagetitle}
\begin{plotpagetitle}{}{$k=3$, $\tau^2 \sim h^3$}
\ploteUt{figures/standing/data/rho0.02_p2_sts/data.txt} \plotpin{0.2}{0.5cm}{-90}{$e_t(\sfU)$}{black};
\ploteut{figures/standing/data/rho0.02_p2_sts/data.txt} \plotpin{0.4}{0.5cm}{-90}{$e_t(u)$}{red};
\plotewt{figures/standing/data/rho0.02_p2_sts/data.txt} \plotpin{0.5}{0.5cm}{ 45}{$e_t(w)$}{blue};
\SlopeTriangle{0.5}{-0.1}{0.15}{-3}{$h^3$}{}
\end{plotpagetitle}
\caption{Standing wave: reconstruction errors on the time derivative}
\label{figure_reconstruction_time}
\end{figure}

\subsection{Properties of the error estimator}

We now focus on the properties of the error estimator,
and we start with the standing wave example. We first study
short-time error control by setting $\rho := 1$. The corresponding results are reported
in Figure~\ref{figure_standing_short}. The first observation is that 
the data-oscillation term $\eta_f$ always superconverges. Besides,
asymptotically, we indeed obtain a guaranteed error upper bound, and the effectivity
index is at most 3 in all cases. Moreover, for $k=1$, the estimator seems to be asymptotically
exact. Since the leapfrog scheme provides a time discretization error decaying as 
$\tau^2 \sim h^2$,
this is to be expected, and it is in agreement with the results observed in the space
semi-discrete case in \cite{Theo:23}. Similarly, we see that, for $k=1$, the error estimator 
$M$ superconverges.

\begin{figure}
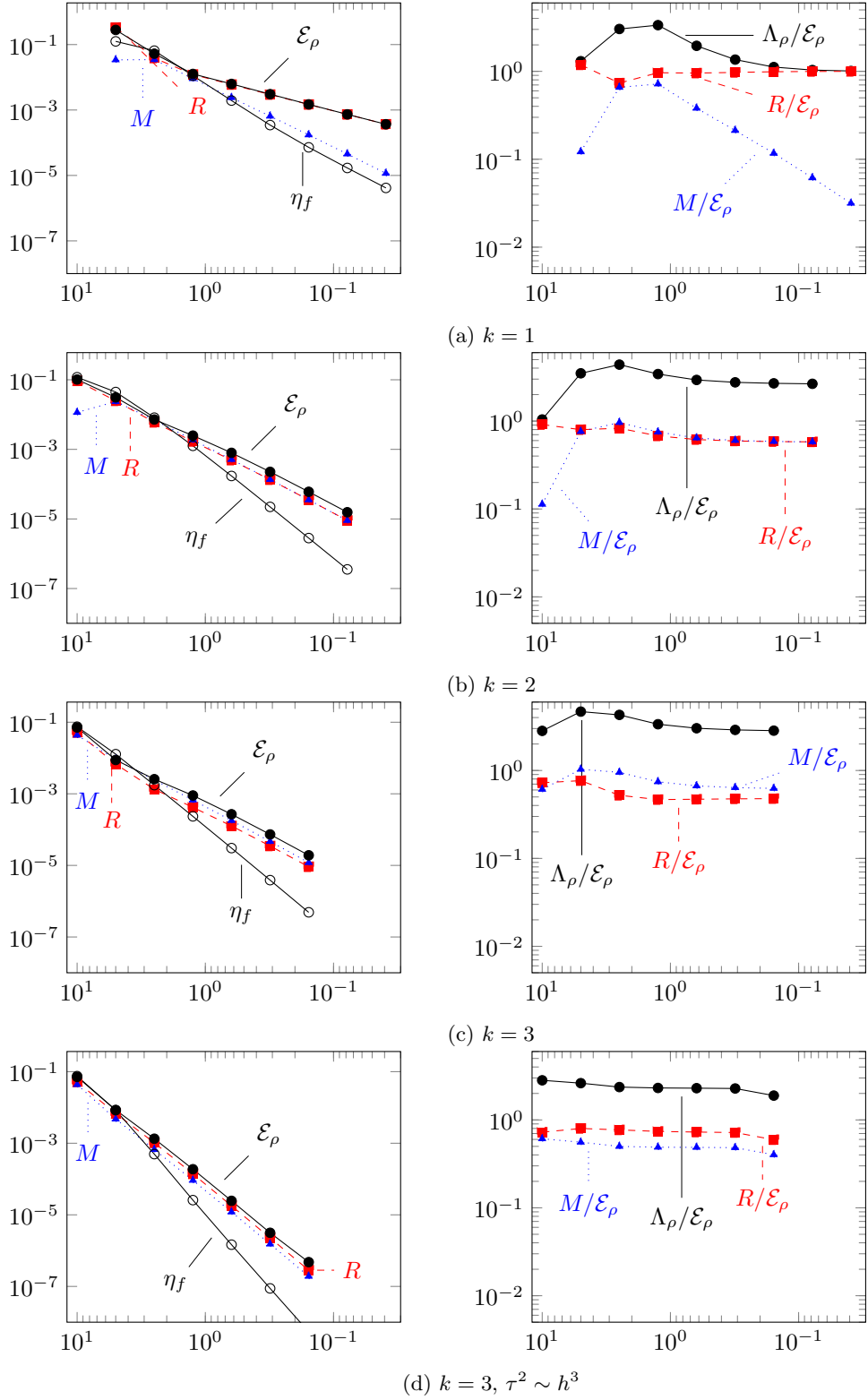

{
\begin{plotpage}
{
	ymax=0,
	ymin=1.e-8
}
\plotR{figures/standing/data/rho1.00_p0_r0.9/data.txt} \plotpin{0.05}{1cm}{-45}{$R$}{red};
\plotM{figures/standing/data/rho1.00_p0_r0.9/data.txt} \plotpin{0.05}{0.5cm}{-90}{$M$}{blue};
\plotF{figures/standing/data/rho1.00_p0_r0.9/data.txt} \plotpin{0.7}{0.5cm}{-90}{$\eta_f$}{black};
\plotE{figures/standing/data/rho1.00_p0_r0.9/data.txt} \plotpin{0.6}{0.5cm}{45}{$\calE_\rho$}{black};
\end{plotpage}
\begin{plotpage}
{
	ymin = 0.005,
	ymax = 6.000
}
\plotMRE{figures/standing/data/rho1.00_p0_r0.9/data.txt} \plotpin{0.4}{1cm}{0}{$\Lambda_\rho/\calE_\rho$}{black};
\plotRE {figures/standing/data/rho1.00_p0_r0.9/data.txt} \plotpin{0.4}{1cm}{-10}{$R/\calE_\rho$}{red};
\plotME {figures/standing/data/rho1.00_p0_r0.9/data.txt} \plotpin{0.7}{0.5cm}{-120}{$M/\calE_\rho$}{blue};
\end{plotpage}
\subcaption{$k=1$}

\begin{plotpage}
{
	ymax=0,
	ymin=1.e-8
}
\plotR{figures/standing/data/rho1.00_p1_r0.9/data.txt} \plotpin{0.2}{0.5cm}{-90}{$R$}{red};
\plotM{figures/standing/data/rho1.00_p1_r0.9/data.txt} \plotpin{0.05}{0.5cm}{-90}{$M$}{blue};
\plotF{figures/standing/data/rho1.00_p1_r0.9/data.txt} \plotpin{0.6}{0.5cm}{-135}{$\eta_f$}{black};
\plotE{figures/standing/data/rho1.00_p1_r0.9/data.txt} \plotpin{0.6}{0.5cm}{45}{$\calE_\rho$}{black};
\end{plotpage}
\begin{plotpage}
{
	ymin = 0.005,
	ymax = 6.000
}
\plotMRE{figures/standing/data/rho1.00_p1_r0.9/data.txt} \plotpin{0.6}{1.5cm}{-90}{$\Lambda_\rho/\calE_\rho$}{black};
\plotRE {figures/standing/data/rho1.00_p1_r0.9/data.txt} \plotpin{0.9}{1cm}{-90}{$R/\calE_\rho$}{red};
\plotME {figures/standing/data/rho1.00_p1_r0.9/data.txt} \plotpin{0.1}{0.5cm}{-70}{$M/\calE_\rho$}{blue};
\end{plotpage}
\subcaption{$k=2$}

\begin{plotpage}
{
	ymax=0,
	ymin=1.e-8
}
\plotR{figures/standing/data/rho1.00_p2_r0.9/data.txt} \plotpin{0.2}{0.5cm}{-90}{$R$}{red};
\plotM{figures/standing/data/rho1.00_p2_r0.9/data.txt} \plotpin{0.05}{0.5cm}{-90}{$M$}{blue};
\plotF{figures/standing/data/rho1.00_p2_r0.9/data.txt} \plotpin{0.7}{0.5cm}{-90}{$\eta_f$}{black};
\plotE{figures/standing/data/rho1.00_p2_r0.9/data.txt} \plotpin{0.6}{0.5cm}{45}{$\calE_\rho$}{black};
\end{plotpage}
\begin{plotpage}
{
	ymin = 0.005,
	ymax = 6.000
}
\plotRE {figures/standing/data/rho1.00_p2_r0.9/data.txt} \plotpin{0.6}{0.5cm}{-90}{$R/\calE_\rho$}{red};
\plotME {figures/standing/data/rho1.00_p2_r0.9/data.txt} \plotpin{0.85}{0.5cm}{10}{$M/\calE_\rho$}{blue};
\plotMRE{figures/standing/data/rho1.00_p2_r0.9/data.txt} \plotpin{0.2}{2.0cm}{-90}{$\Lambda_\rho/\calE_\rho$}{black};
\end{plotpage}
\subcaption{$k=3$}

\begin{plotpage}
{
	ymax=0,
	ymin=1.e-8
}
\plotR{figures/standing/data/rho1.00_p2_sts/data.txt} \plotpin{1.00}{0.25cm}{0}{$R$}{red};
\plotM{figures/standing/data/rho1.00_p2_sts/data.txt} \plotpin{0.05}{0.5cm}{-90}{$M$}{blue};
\plotF{figures/standing/data/rho1.00_p2_sts/data.txt} \plotpin{0.6}{0.5cm}{-120}{$\eta_f$}{black};
\plotE{figures/standing/data/rho1.00_p2_sts/data.txt} \plotpin{0.6}{0.5cm}{45}{$\calE_\rho$}{black};
\end{plotpage}
\begin{plotpage}
{
	ymin = 0.005,
	ymax = 6.000
}
\plotRE {figures/standing/data/rho1.00_p2_sts/data.txt} \plotpin{0.95}{0.5cm}{-90}{$R/\calE_\rho$}{red};
\plotME {figures/standing/data/rho1.00_p2_sts/data.txt} \plotpin{0.2}{0.5cm}{-90}{$M/\calE_\rho$}{blue};
\plotMRE{figures/standing/data/rho1.00_p2_sts/data.txt} \plotpin{0.6}{1.5cm}{-90}{$\Lambda_\rho/\calE_\rho$}{black};
\end{plotpage}
\subcaption{$k=3$, $\tau^2 \sim h^3$}
}

\caption{Standing wave example with $\rho=1$}
\label{figure_standing_short}
\end{figure}

Still for the standing wave example, we now address long-time error control by
setting $\rho := 0.02$. From the results reported in Figure~\ref{figure_standing_long},
we again observe a superconvergence of the data oscillation term
$\eta_f$ for all the values of $k$. Asymptotically, the bound given by $\eta$ is 
indeed guaranteed, and the effectivity index is about 10 at worst.
For $k=1$, we see that $M$ is in fact the dominant part of the estimator, meaning
that the asymptotic regime where the time discretization error becomes negligible
has not yet been reached. This is in agreement with the above observation regarding
the convergence rates.
Interestingly, the cases where $k \geq 2$ show that the $M$ component of the
estimator is indeed required to obtain a guaranteed error upper bound. Indeed, in those cases,
$R/\calE_\rho \ll 1$, even asymptotically. The question of whether the factor $\sqrt{10}$
in the bound is sharp remains open.

\begin{figure}
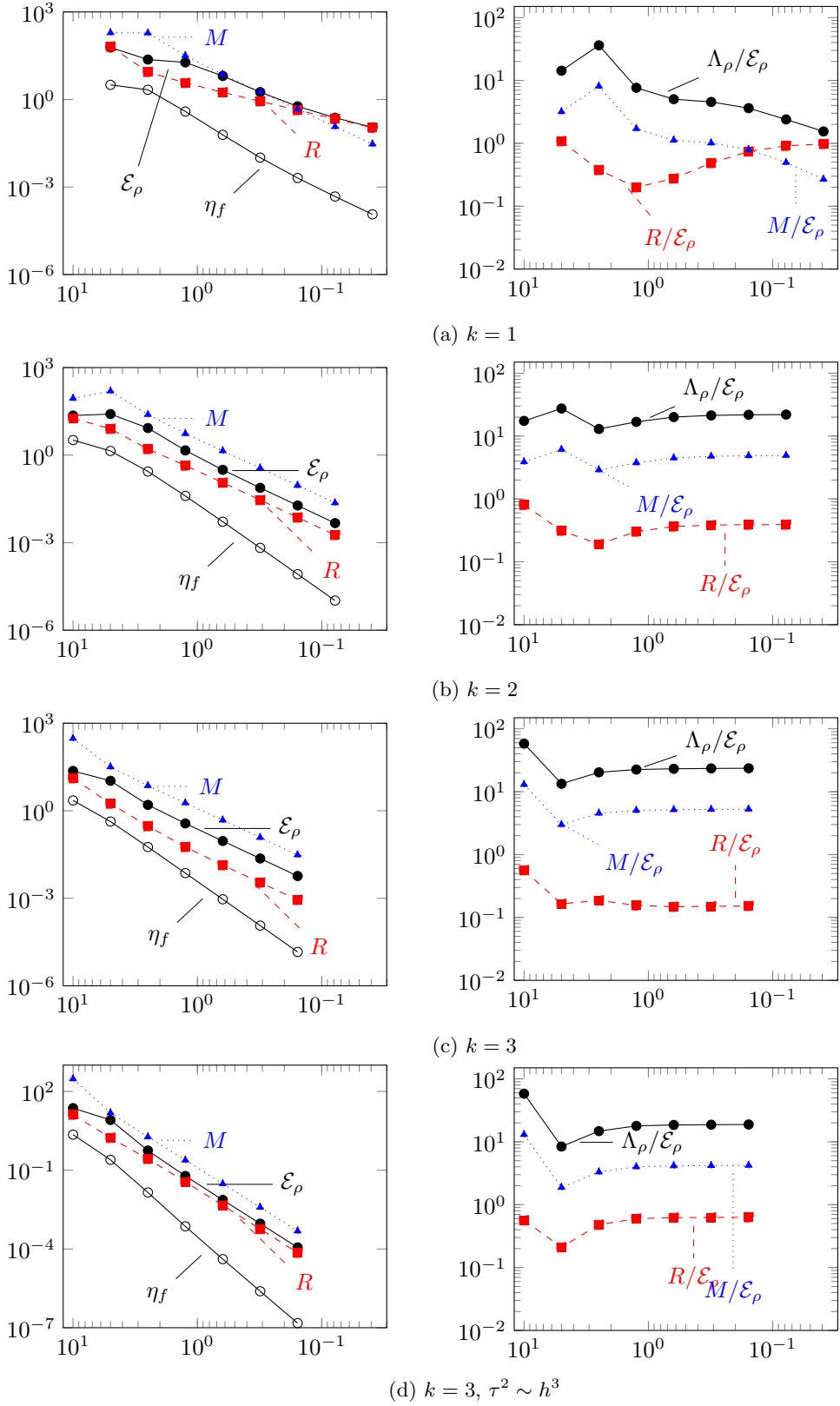

{
\begin{plotpage}
{
	ymin = 1.e-6,
	ymax = 1.e+3
}
\plotE{figures/standing/data/rho0.02_p0_r0.9/data.txt} \plotpin{0.2}{1.5cm}{-100}{$\calE_\rho$}{black};
\plotR{figures/standing/data/rho0.02_p0_r0.9/data.txt} \plotpin{0.65}{0.5cm}{-45}{$R$}{red};
\plotM{figures/standing/data/rho0.02_p0_r0.9/data.txt} \plotpin{0.1}{0.5cm}{0}{$M$}{blue};
\plotF{figures/standing/data/rho0.02_p0_r0.9/data.txt} \plotpin{0.6}{0.5cm}{-135}{$\eta_f$}{black};
\end{plotpage}
\begin{plotpage}
{
	ymin =   0.01,
	ymax = 150.00
}
\plotMRE{figures/standing/data/rho0.02_p0_r0.9/data.txt} \plotpin{0.5}{0.5cm}{20}{$\Lambda_\rho/\calE_\rho$}{black};
\plotRE {figures/standing/data/rho0.02_p0_r0.9/data.txt} \plotpin{0.3}{0.5cm}{-70}{$R/\calE_\rho$}{red};
\plotME {figures/standing/data/rho0.02_p0_r0.9/data.txt} \plotpin{0.90}{0.5cm}{-90}{$M/\calE_\rho$}{blue};
\end{plotpage}
\subcaption{$k=1$}

\begin{plotpage}
{
	ymin = 1.e-6,
	ymax = 1.e+3
}
\plotE{figures/standing/data/rho0.02_p1_r0.9/data.txt} \plotpin{0.55}{1cm}{0}{$\calE_\rho$}{black};
\plotR{figures/standing/data/rho0.02_p1_r0.9/data.txt} \plotpin{0.70}{1cm}{-45}{$R$}{red};
\plotM{figures/standing/data/rho0.02_p1_r0.9/data.txt} \plotpin{0.3}{0.5cm}{0}{$M$}{blue};
\plotF{figures/standing/data/rho0.02_p1_r0.9/data.txt} \plotpin{0.6}{0.5cm}{-135}{$\eta_f$}{black};
\end{plotpage}
\begin{plotpage}
{
	ymin =   0.01,
	ymax = 150.00
}
\plotMRE{figures/standing/data/rho0.02_p1_r0.9/data.txt} \plotpin{0.5}{0.4cm}{ 20}{$\Lambda_\rho/\calE_\rho$}{black};
\plotRE {figures/standing/data/rho0.02_p1_r0.9/data.txt} \plotpin{0.8}{0.5cm}{-90}{$R/\calE_\rho$}{red};
\plotME {figures/standing/data/rho0.02_p1_r0.9/data.txt} \plotpin{0.3}{0.5cm}{-30}{$M/\calE_\rho$}{blue};
\end{plotpage}
\subcaption{$k=2$}

\begin{plotpage}
{
	ymin = 1.e-6,
	ymax = 1.e+3
}
\plotE{figures/standing/data/rho0.02_p2_r0.9/data.txt} \plotpin{0.55}{1cm}{0}{$\calE_\rho$}{black};
\plotR{figures/standing/data/rho0.02_p2_r0.9/data.txt} \plotpin{0.80}{1cm}{-45}{$R$}{red};
\plotM{figures/standing/data/rho0.02_p2_r0.9/data.txt} \plotpin{0.4}{0.5cm}{  0}{$M$}{blue};
\plotF{figures/standing/data/rho0.02_p2_r0.9/data.txt} \plotpin{0.6}{0.5cm}{-135}{$\eta_f$}{black};
\end{plotpage}
\begin{plotpage}
{
	ymin =   0.01,
	ymax = 150.00
}
\plotMRE{figures/standing/data/rho0.02_p2_r0.9/data.txt} \plotpin{0.6}{0.5cm}{ 10}{$\Lambda_\rho/\calE_\rho$}{black};
\plotRE {figures/standing/data/rho0.02_p2_r0.9/data.txt} \plotpin{0.95}{0.5cm}{ 90}{$R/\calE_\rho$}{red};
\plotME {figures/standing/data/rho0.02_p2_r0.9/data.txt} \plotpin{0.3}{0.5cm}{-30}{$M/\calE_\rho$}{blue};
\end{plotpage}
\subcaption{$k=3$}

\begin{plotpage}
{
	ymin = 1.e-7,
	ymax = 1.e+3
}
\plotE{figures/standing/data/rho0.02_p2_sts/data.txt} \plotpin{0.55}{1cm}{0}{$\calE_\rho$}{black};
\plotR{figures/standing/data/rho0.02_p2_sts/data.txt} \plotpin{0.70}{1cm}{-45}{$R$}{red};
\plotM{figures/standing/data/rho0.02_p2_sts/data.txt} \plotpin{0.4}{0.5cm}{  0}{$M$}{blue};
\plotF{figures/standing/data/rho0.02_p2_sts/data.txt} \plotpin{0.6}{0.5cm}{-135}{$\eta_f$}{black};
\end{plotpage}
\begin{plotpage}
{
	ymin =   0.01,
	ymax = 150.00
}
\plotRE {figures/standing/data/rho0.02_p2_sts/data.txt} \plotpin{0.80}{0.5cm}{-90}{$R/\calE_\rho$}{red};
\plotME {figures/standing/data/rho0.02_p2_sts/data.txt} \plotpin{0.95}{1.5cm}{-90}{$M/\calE_\rho$}{blue};
\plotMRE{figures/standing/data/rho0.02_p2_sts/data.txt} \plotpin{0.40}{0.5cm}{  0}{$\Lambda_\rho/\calE_\rho$}{black};
\end{plotpage}
\subcaption{$k=3$, $\tau^2 \sim h^3$}
}

\caption{Standing wave example with $\rho=0.02$}
\label{figure_standing_long}
\end{figure}

We now consider the propagating wave example. Addressing first short-time error control,
we set $\rho:=0.2$. Results are reported in Figure~\ref{figure_propagating_short}.
Asymptotically, we obtain guaranteed error bounds
with effectivity indices of $10$ at worst.

\begin{figure}
{
\begin{plotpage}
{
	ymin = 1.e-7,
	ymax = 10
}
\plotR{figures/propagating/data/rho0.20_p0_r0.9/data.txt} \plotpin{0.5}{1.0cm}{-135}{$R$}{red};
\plotM{figures/propagating/data/rho0.20_p0_r0.9/data.txt} \plotpin{0.15}{0.5cm}{0}{$M$}{blue};
\plotF{figures/propagating/data/rho0.20_p0_r0.9/data.txt} \plotpin{0.6}{0.5cm}{-135}{$\eta_f$}{black};
\plotE{figures/propagating/data/rho0.20_p0_r0.9/data.txt} \plotpin{0.55}{0.5cm}{10}{$\calE_\rho$}{black};
\end{plotpage}
\begin{plotpage}
{
	ymin =  0.02,
	ymax = 40.00
}
\plotRE{figures/propagating/data/rho0.20_p0_r0.9/data.txt} \plotpin{0.3}{0.5cm}{-90}{$R/\calE_\rho$}{red};
\plotME{figures/propagating/data/rho0.20_p0_r0.9/data.txt} \plotpin{0.8}{0.5cm}{-110}{$M/\calE_\rho$}{blue};
\plotMRE{figures/propagating/data/rho0.20_p0_r0.9/data.txt} \plotpin{0.4}{0.5cm}{10}{$\Lambda_\rho/\calE_\rho$}{black};
\end{plotpage}
\subcaption{$k=1$}

\begin{plotpage}
{
	ymin = 1.e-7,
	ymax = 10
}
\plotR{figures/propagating/data/rho0.20_p1_r0.9/data.txt} \plotpin{0.70}{1cm}{-45}{$R$}{red};
\plotM{figures/propagating/data/rho0.20_p1_r0.9/data.txt} \plotpin{0.3}{0.5cm}{0}{$M$}{blue};
\plotF{figures/propagating/data/rho0.20_p1_r0.9/data.txt} \plotpin{0.6}{0.5cm}{-135}{$\eta_f$}{black};
\plotE{figures/propagating/data/rho0.20_p1_r0.9/data.txt} \plotpin{0.55}{1cm}{0}{$\calE_\rho$}{black};
\end{plotpage}
\begin{plotpage}
{
	ymin =   0.02,
	ymax = 40.00
}
\plotRE{figures/propagating/data/rho0.20_p1_r0.9/data.txt} \plotpin{0.8}{0.5cm}{-90}{$R/\calE_\rho$}{red};
\plotME{figures/propagating/data/rho0.20_p1_r0.9/data.txt} \plotpin{0.4}{1.5cm}{-90}{$M/\calE_\rho$}{blue};
\plotMRE{figures/propagating/data/rho0.20_p1_r0.9/data.txt} \plotpin{0.5}{0.5cm}{ 20}{$\Lambda_\rho/\calE_\rho$}{black};
\end{plotpage}
\subcaption{$k=2$}

\begin{plotpage}
{
	ymin = 1.e-7,
	ymax = 10
}
\plotR{figures/propagating/data/rho0.20_p2_r0.9/data.txt} \plotpin{0.70}{1cm}{-45}{$R$}{red};
\plotM{figures/propagating/data/rho0.20_p2_r0.9/data.txt} \plotpin{0.4}{0.5cm}{  0}{$M$}{blue};
\plotF{figures/propagating/data/rho0.20_p2_r0.9/data.txt} \plotpin{0.6}{0.5cm}{-135}{$\eta_f$}{black};
\plotE{figures/propagating/data/rho0.20_p2_r0.9/data.txt} \plotpin{0.55}{1cm}{0}{$\calE_\rho$}{black};
\end{plotpage}
\begin{plotpage}
{
	ymin =   0.02,
	ymax = 40.00
}
\plotRE{figures/propagating/data/rho0.20_p2_r0.9/data.txt} \plotpin{0.3}{0.5cm}{-90}{$R/\calE_\rho$}{red};
\plotME{figures/propagating/data/rho0.20_p2_r0.9/data.txt} \plotpin{0.6}{0.5cm}{-45}{$M/\calE_\rho$}{blue};
\plotMRE{figures/propagating/data/rho0.20_p2_r0.9/data.txt} \plotpin{0.9}{0.5cm}{-45}{$\Lambda_\rho/\calE_\rho$}{black};
\end{plotpage}
\subcaption{$k=3$}

\begin{plotpage}
{
	ymin = 1.e-7,
	ymax = 10
}
\plotR{figures/propagating/data/rho0.20_p2_sts/data.txt} \plotpin{0.15}{1.0cm}{-110}{$R$}{red};
\plotM{figures/propagating/data/rho0.20_p2_sts/data.txt} \plotpin{0.00}{0.5cm}{   0}{$M$}{blue};
\plotF{figures/propagating/data/rho0.20_p2_sts/data.txt} \plotpin{0.60}{0.5cm}{-135}{$\eta_f$}{black};
\plotE{figures/propagating/data/rho0.20_p2_sts/data.txt} \plotpin{0.25}{0.5cm}{   0}{$\calE_\rho$}{black};
\end{plotpage}
\begin{plotpage}
{
	ymin =   0.02,
	ymax = 40.00
}
\plotRE{figures/propagating/data/rho0.20_p2_sts/data.txt} \plotpin{0.2}{0.5cm}{-90}{$R/\calE_\rho$}{red};
\plotME{figures/propagating/data/rho0.20_p2_sts/data.txt} \plotpin{0.6}{0.5cm}{-45}{$M/\calE_\rho$}{blue};
\plotMRE{figures/propagating/data/rho0.20_p2_sts/data.txt} \plotpin{0.9}{0.5cm}{45}{$\Lambda_\rho/\calE_\rho$}{black};
\end{plotpage}
\subcaption{$k=3$, $\tau^2 \sim h^3$}
}

\caption{Propagating wave example with $\rho=0.2$}
\label{figure_propagating_short}
\end{figure}

Finally, to study long-time error control, we set $\rho:=0.01$. 
Results are reported in Figure~\ref{figure_propagating_long}.
We obtain asymptotically a guaranteed error upper bound
with effectivity indices of about $10$, except for $k=1$ where the asymptotic
regime is not yet reached.

\begin{figure}
{
\begin{plotpage}
{
	ymin = 1.e-6,
	ymax = 1.e+3
}
\plotE{figures/propagating/data/rho0.01_p0_r0.9/data.txt} \plotpin{0.4}{1cm}{10}{$\calE_\rho$}{black};
\plotR{figures/propagating/data/rho0.01_p0_r0.9/data.txt} \plotpin{0.65}{0.5cm}{-45}{$R$}{red};
\plotM{figures/propagating/data/rho0.01_p0_r0.9/data.txt} \plotpin{0.2}{0.5cm}{0}{$M$}{blue};
\plotF{figures/propagating/data/rho0.01_p0_r0.9/data.txt} \plotpin{0.6}{0.5cm}{-135}{$\eta_f$}{black};
\end{plotpage}
\begin{plotpage}
{
	ymin =   0.02,
	ymax = 300.00
}
\plotMRE{figures/propagating/data/rho0.01_p0_r0.9/data.txt} \plotpin{0.5}{0.5cm}{0}{$\Lambda_\rho/\calE_\rho$}{black};
\plotRE{figures/propagating/data/rho0.01_p0_r0.9/data.txt} \plotpin{0.3}{0.5cm}{-70}{$R/\calE_\rho$}{red};
\plotME{figures/propagating/data/rho0.01_p0_r0.9/data.txt} \plotpin{0.35}{0.5cm}{-110}{$M/\calE_\rho$}{blue};
\end{plotpage}
\subcaption{$k=1$}

\begin{plotpage}
{
	ymin = 1.e-6,
	ymax = 1.e+3
}
\plotE{figures/propagating/data/rho0.01_p1_r0.9/data.txt} \plotpin{0.55}{1cm}{0}{$\calE_\rho$}{black};
\plotR{figures/propagating/data/rho0.01_p1_r0.9/data.txt} \plotpin{0.55}{1cm}{-45}{$R$}{red};
\plotM{figures/propagating/data/rho0.01_p1_r0.9/data.txt} \plotpin{0.3}{0.5cm}{0}{$M$}{blue};
\plotF{figures/propagating/data/rho0.01_p1_r0.9/data.txt} \plotpin{0.6}{0.5cm}{-135}{$\eta_f$}{black};
\end{plotpage}
\begin{plotpage}
{
	ymin =   0.02,
	ymax = 300.00
}
\plotMRE{figures/propagating/data/rho0.01_p1_r0.9/data.txt} \plotpin{0.5}{0.5cm}{ 20}{$\Lambda_\rho/\calE_\rho$}{black};
\plotRE{figures/propagating/data/rho0.01_p1_r0.9/data.txt} \plotpin{0.8}{0.5cm}{ 90}{$R/\calE_\rho$}{red};
\plotME{figures/propagating/data/rho0.01_p1_r0.9/data.txt} \plotpin{0.3}{0.5cm}{-90}{$M/\calE_\rho$}{blue};
\end{plotpage}
\subcaption{$k=2$}

\begin{plotpage}
{
	ymin = 1.e-6,
	ymax = 1.e+3
}
\plotE{figures/propagating/data/rho0.01_p2_r0.9/data.txt} \plotpin{0.55}{1cm}{0}{$\calE_\rho$}{black};
\plotR{figures/propagating/data/rho0.01_p2_r0.9/data.txt} \plotpin{0.70}{1cm}{-45}{$R$}{red};
\plotM{figures/propagating/data/rho0.01_p2_r0.9/data.txt} \plotpin{0.4}{0.5cm}{  0}{$M$}{blue};
\plotF{figures/propagating/data/rho0.01_p2_r0.9/data.txt} \plotpin{0.6}{0.5cm}{-135}{$\eta_f$}{black};
\end{plotpage}
\begin{plotpage}
{
	ymin =   0.02,
	ymax = 300.00
}
\plotMRE{figures/propagating/data/rho0.01_p2_r0.9/data.txt} \plotpin{0.6}{0.5cm}{ 10}{$\Lambda_\rho/\calE_\rho$}{black};
\plotRE{figures/propagating/data/rho0.01_p2_r0.9/data.txt} \plotpin{0.95}{0.5cm}{ 90}{$R/\calE_\rho$}{red};
\plotME{figures/propagating/data/rho0.01_p2_r0.9/data.txt} \plotpin{0.6}{0.5cm}{-90}{$M/\calE_\rho$}{blue};
\end{plotpage}
\subcaption{$k=3$}

\begin{plotpage}
{
	ymin = 1.e-6,
	ymax = 1.e+3
}
\plotE{figures/propagating/data/rho0.01_p2_sts/data.txt} \plotpin{0.55}{1cm}{0}{$\calE_\rho$}{black};
\plotR{figures/propagating/data/rho0.01_p2_sts/data.txt} \plotpin{0.70}{1cm}{-45}{$R$}{red};
\plotM{figures/propagating/data/rho0.01_p2_sts/data.txt} \plotpin{0.4}{0.5cm}{  0}{$M$}{blue};
\plotF{figures/propagating/data/rho0.01_p2_sts/data.txt} \plotpin{0.6}{0.5cm}{-135}{$\eta_f$}{black};
\end{plotpage}
\begin{plotpage}
{
	ymin =   0.02,
	ymax = 300.00
}
\plotMRE{figures/propagating/data/rho0.01_p2_sts/data.txt} \plotpin{0.60}{0.5cm}{ 90}{$\Lambda_\rho/\calE_\rho$}{black};
\plotRE {figures/propagating/data/rho0.01_p2_sts/data.txt} \plotpin{0.95}{0.5cm}{-90}{$R/\calE_\rho$}{red};
\plotME {figures/propagating/data/rho0.01_p2_sts/data.txt} \plotpin{0.60}{0.5cm}{  0}{$M/\calE_\rho$}{blue};
\end{plotpage}
\subcaption{$k=3$, $\tau^2 \sim h^3$}
}

\caption{Propagating wave example with $\rho=0.01$}
\label{figure_propagating_long}
\end{figure}

\subsection{Time-step variation}

Finally, we investigate the effect of reducing the time step on fixed meshes
for $k\in\{2,3\}$.
\rev{The goal here is to numerically illustrate that for a small time step,
the proposed estimator for the fully discrete scheme approaches the estimator
of~\cite{Theo:23} when the time-discretization error becomes negligible.}
To do so, for quadratic elements, we fix
$\tau$ using the CFL condition with $r=0.9,0.8,0.5,0.2$, and $0.1$.
For cubic elements, we select $\tau^2 \sim h^3$ as described above
with $r_0=0.9,0.8,0.5,0.2$, and $0.1$. We set $\rho:=0.05$.
As can be seen in Figure \ref{figure_time_step}, reducing the
time step has little effect on the actual error itself. However,
it does improve the efficiency of the estimator. In particular,
we obtain asymptotically constant-free error upper bounds when $\tau$ is
selected ten times smaller than the CFL constraint. 
%This is again in agreement with the results in \cite{Theo:23} for the semi-discrete case.

\begin{figure}
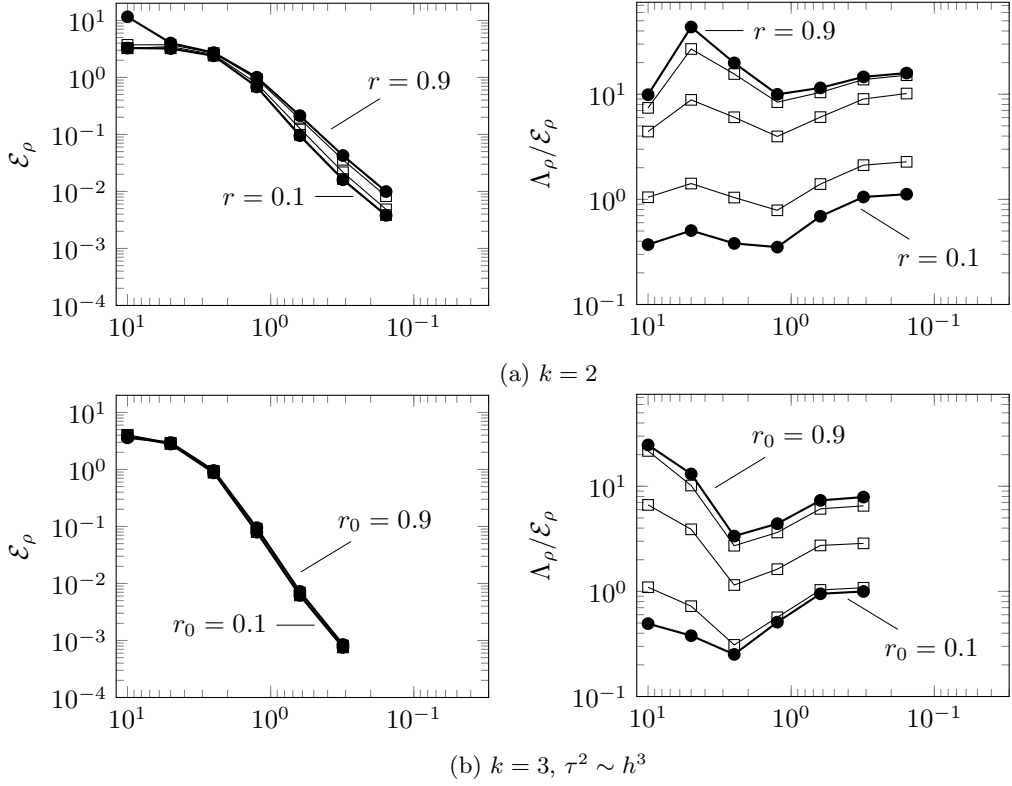

{
\begin{plotpage}
{
	ymin=1.e-4,
	ymax=20,
	ylabel=$\calE_\rho$
}
\plotCFLE{figures/propagating/data/rho0.05_p1/data_0.800.txt}{thin,mark=square};
\plotCFLE{figures/propagating/data/rho0.05_p1/data_0.500.txt}{thin,mark=square};
\plotCFLE{figures/propagating/data/rho0.05_p1/data_0.200.txt}{thin,mark=square};
\plotCFLE{figures/propagating/data/rho0.05_p1/data_0.900.txt}{thick,mark=*}
\plotpin{0.7}{0.5cm}{ 45}{$r = 0.9$}{black};
\plotCFLE{figures/propagating/data/rho0.05_p1/data_0.100.txt}{thick,mark=*}
\plotpin{0.9}{0.5cm}{180}{$r = 0.1$}{black};
\end{plotpage}
\begin{plotpage}
{
	ymin=0.1,
	ymax=75,
	ylabel=$\Lambda_\rho/\calE_\rho$
}
\plotCFLMRE{figures/propagating/data/rho0.05_p1/data_0.800.txt}{thin,mark=square};
\plotCFLMRE{figures/propagating/data/rho0.05_p1/data_0.500.txt}{thin,mark=square};
\plotCFLMRE{figures/propagating/data/rho0.05_p1/data_0.200.txt}{thin,mark=square};
\plotCFLMRE{figures/propagating/data/rho0.05_p1/data_0.900.txt}{thick,mark=*}
\plotpin{0.3}{0.5cm}{  0}{$r = 0.9$}{black};
\plotCFLMRE{figures/propagating/data/rho0.05_p1/data_0.100.txt}{thick,mark=*}
\plotpin{0.8}{0.5cm}{-45}{$r = 0.1$}{black};
\end{plotpage}
\subcaption{$k=2$}

\begin{plotpage}
{
	ymin=1.e-4,
	ymax=20,
	ylabel=$\calE_\rho$
}
\plotCFLE{figures/propagating/data/rho0.05_p2/data_0.800.txt}{thin,mark=square};
\plotCFLE{figures/propagating/data/rho0.05_p2/data_0.500.txt}{thin,mark=square};
\plotCFLE{figures/propagating/data/rho0.05_p2/data_0.200.txt}{thin,mark=square};
\plotCFLE{figures/propagating/data/rho0.05_p2/data_0.900.txt}{thick,mark=*}
\plotpin{0.7}{0.5cm}{ 45}{$r_0 = 0.9$}{black};
\plotCFLE{figures/propagating/data/rho0.05_p2/data_0.100.txt}{thick,mark=*}
\plotpin{0.9}{0.5cm}{180}{$r_0 = 0.1$}{black};
\end{plotpage}
\begin{plotpage}
{
	ymin=0.1,
	ymax=75,
	ylabel=$\Lambda_\rho/\calE_\rho$
}
\plotCFLMRE{figures/propagating/data/rho0.05_p2/data_0.800.txt}{thin,mark=square};
\plotCFLMRE{figures/propagating/data/rho0.05_p2/data_0.500.txt}{thin,mark=square};
\plotCFLMRE{figures/propagating/data/rho0.05_p2/data_0.200.txt}{thin,mark=square};
\plotCFLMRE{figures/propagating/data/rho0.05_p2/data_0.900.txt}{thick,mark=*}
\plotpin{0.3}{0.5cm}{45}{$r_0 = 0.9$}{black};
\plotCFLMRE{figures/propagating/data/rho0.05_p2/data_0.100.txt}{thick,mark=*}
\plotpin{0.9}{0.5cm}{-45}{$r_0 = 0.1$}{black};
\end{plotpage}
\subcaption{$k=3$, $\tau^2 \sim h^3$}
}

\caption{Time-step variation in the propagating wave example, $\rho=0.05$}
\label{figure_time_step}
\end{figure}

\section{Damped stability \rev{and a priori error} estimates}
\label{section_stability_leapfrog}

This section collects some stability \rev{and error} estimates for the leapfrog scheme in the damped
energy norm. The \rev{stability} estimates are essentially a discrete counterpart of the a priori estimate
established in Lemma~\ref{lem:a_priori} at the continuous level. They are established
by adapting known arguments devised for leapfrog stability estimates in the classical energy norm.

\subsection{Abstract estimate}
\label{sec:abstract}

Let $(\sfG^n)_{n\in\polN}$ be a sequence in $L^2(\Omega)$.
Let $(\sfX^n)_{n\in\polN}\subset V_h^\polN$ be such that 
\begin{equation} \label{eq:leapfrog_X}
\frac{1}{\tau^{2}}(\sfX^{n+1}-2\sfX^n+\sfX^{n-1},v_h)_\Omega 
+ (\nabla \sfX^n,\nabla v_h)_\Omega = (\sfG^n,v_h)_\Omega \quad
\forall n\ge 0,\; \forall v_h\in V_h,
\end{equation}
with the initial condition $\sfX^0=\sfX^{-1}=0$.
Notice that here we do not assume that $\sfG^0=0$, and consequently we can have $\sfX^1\ne0$.
For all $n\ge-1$, we define the following midpoint state and velocity:
\begin{equation}
\sfX\nmz := \frac12 (\sfX^{n+1}+\sfX^n), \quad
\dsfX\nmz := \frac{1}{\tau}(\sfX^{n+1}-\sfX^n),
\end{equation}
as well as the discrete energy functional
\begin{equation} \label{eq:def_disc_engy}
E_\sfX\nmz := m_{h\tau}(\dsfX\nmz,\dsfX\nmz) + \|\nabla \sfX\nmz\|^2_\Omega,
\end{equation}
with the bilinear form $m_{h\tau}$ defined in~\eqref{eq:def_mht}. Recall that
the discrete energy functional 
defines a quadratic form on $\dsfX\nmz$ and $\nabla \sfX\nmz$
under the CFL condition~\eqref{eq:CFL} since \eqref{eq:mht_L2} implies that
\begin{equation} \label{eq:tildeE}
E_\sfX\nmz \ge \mu_0 \|\dsfX\nmz\|^2_\Omega + \|\nabla \sfX\nmz\|^2_\Omega.
\end{equation}
with $\mu_0\in (0,1)$.

\begin{lemma}[Damped stability for leapfrog scheme] \label{lem:damped_leapfrog}
Assume the CFL condition~\eqref{eq:CFL}. 
Let $(\sfX^n)_{n\in\polN}$ solve~\eqref{eq:leapfrog_X} with $\sfX^0=\sfX^{-1}=0$,
and let the discrete energy
functional $E_\sfX\nmz$ be defined in~\eqref{eq:def_disc_engy}. 
Let $\theta:=\rho \tau \in (0,1]$ and set
$C_\sfX(\theta) := \frac{169}{12}(13-6\theta)^{-1}\theta(1-e^{-\theta})^{-1}$. The following holds:
\begin{equation} \label{eq:disc_stab_leapfrog}
\sum_{n\in\polN} \tau E_\sfX\nmz e^{-2\rho t^n}
\le C_\sfX(\theta) \mu_0^{-1} \frac{1}{\rho^2} \sum_{n\in\polN} \tau 
\|\sfG^n\|_\Omega^2 e^{-2\rho t^n}.
\end{equation}
\end{lemma}

\begin{proof}
The starting point is the following well-known energy identity for the leapfrog scheme: 
\[
E_\sfX\nmz - E_\sfX\nmzm = (\sfG^n,\sfX^{n+1}-\sfX^{n-1})_\Omega
= \tau (\sfG^n,\dsfX\nmz+\dsfX\nmzm)_\Omega \quad \forall n\in\polN.
\]
Invoking the Cauchy--Schwarz and Young's inequality (with parameter $\gamma>\frac14$) together with \eqref{eq:tildeE} gives
\begin{align*}
\tau |(\sfG^n,\dsfX\nmz+\dsfX\nmzm)_\Omega|
&\le \tau \|\sfG^n\|_\Omega\|\dsfX\nmz\|_\Omega + \tau \|\sfG^n\|_\Omega\|\dsfX\nmzm\|_\Omega \\
&\le 2\gamma\mu_0^{-1} \frac{\tau}{\rho}\|\sfG^n\|_\Omega^2 + \frac{1}{4\gamma} \rho \tau \mu_0 \|\dsfX\nmz\|_\Omega^2  + \frac{1}{4\gamma} \rho \tau \mu_0 \|\dsfX\nmzm\|_\Omega^2 \\
&\le 2\gamma\mu_0^{-1} \frac{\tau}{\rho}\|\sfG^n\|_\Omega^2 + \frac{1}{4\gamma} \rho \tau E_\sfX\nmz + \frac{1}{4\gamma} \rho \tau E_\sfX\nmzm.
\end{align*}
Recalling that $\theta := \rho\tau\in (0,1]$ and re-arranging the terms gives
\[
\big( 1 -\tfrac{1}{4\gamma}\theta \big) E_\sfX\nmz \le 2\gamma \mu_0^{-1} \frac{\tau}{\rho}\|\sfG^n\|_\Omega^2 + \big( 1 +\tfrac{1}{4\gamma} \theta\big) E_\sfX\nmzm.
\]
Setting $c_1(\gamma,\theta):=\frac{2\gamma}{1 -\tfrac{1}{4\gamma}\theta}$ and
$c_2(\gamma,\theta):=\tfrac{1+\frac{1}{4\gamma}\theta}{1-\frac{1}{4\gamma}\theta}e^{-\theta}$, this implies that 
\[
E_\sfX\nmz e^{-\rho t^n}\le c_1(\gamma,\theta) \mu_0^{-1} \frac{\tau}{\rho}\|\sfG^n\|_\Omega^2 e^{-\rho t^n} +  c_2(\gamma,\theta) E_\sfX\nmzm e^{-\rho t^{n-1}}.
\]
The ideal choice of $\gamma$ would be to minimize $c_1(\gamma,1)$ while ensuring that $c_2(\gamma,\theta)\le1$ for all $\theta\in (0,1]$. A fairly optimal choice is $\gamma=\frac{13}{24}$ (see Remark~\ref{rem:CX} below for some further discussion). An induction argument gives
\[
E_\sfX\nmz e^{-\rho t^n} \le c_1(\tfrac{13}{24},\theta) \mu_0^{-1} \frac{\tau}{\rho} \sum_{m\in \{0:n\}} \|\sfG^m\|_\Omega^2e^{-\rho t^m},
\]
since $E_\sfX^{-\frac12}=0$.
Summing over all $n\in\polN$ and exchanging the summations on the right-hand side, we infer that
\begin{align*}
\sum_{n\in\polN} \tau E_\sfX\nmz e^{-2\rho t^n} &\le c_1(\tfrac{13}{24},\theta)
\mu_0^{-1} \frac{\tau^2}{\rho} 
\sum_{n\in\polN} e^{-\rho t^n} \sum_{m\in \{0:n\}} \|\sfG^m\|_\Omega^2 e^{-\rho t^m}\\
&= c_1(\tfrac{13}{24},\theta)\mu_0^{-1} \frac{\tau^2}{\rho} \sum_{m\ge0} \bigg(\sum_{n\ge m} e^{-\rho t^n} \bigg)
\|\sfG^m\|_\Omega^2 e^{-\rho t^m} \\
&\le c_1(\tfrac{13}{24},\theta) \theta (1-e^{-\theta})^{-1} \mu_0^{-1} \frac{1}{\rho^2} \sum_{m\ge0} \tau \|\sfG^m\|_\Omega^2 e^{-2\rho t^m},
\end{align*}
since $\sum_{n\ge m} e^{-\rho t^n} = e^{-\rho t^m}(1-e^{-\theta})^{-1}$ and $\theta=\rho\tau$.
\end{proof}

\begin{remark}[Constant $C_\sfX(\theta)$] \label{rem:CX}
We notice that $C_\sfX(\theta)\le \frac{169}{84}\frac{e}{e-1} \approx 3.1828$.
Moreover, we observe that $C_\sfX(\theta)\to C_\sfX(0^+):=
\frac{13}{12}$ as $\theta\to 0^+$. This limit is
relevant as $\theta=\rho\tau$ and $\tau\ll T\le \rho^{-1}$. The exact counterpart of
the a priori estimate~\eqref{eq:key_stab1} would be $C_\sfX(0^+)=1$. This would require taking $\gamma=\frac12$, but then $c_2(\gamma,\theta)$ can take values larger than one; some further optimization should be possible by considering two distinct coefficients in the Young inequalities related to $\dsfX\nmz$ and to $\dsfX\nmzm$, but this is not further explored here.
\end{remark}

\subsection{Application: \rev{stability} estimates on the acceleration}

Let $(\sfU^n)_{n\in\polN}$ solve the leapfrog scheme~\eqref{eq:leapfrog}
with $\sfU^1=\sfU^0=0$. Recall the definition~\eqref{eq:def_sfA}
of the acceleration $\sfA^n$ for all $n\in\polN$, and set
\begin{subequations} \begin{equation} \label{eq:def_sfAnmz}
\sfA\nmz := \frac12(\sfA^{n+1}+\sfA^n), \quad
\dsfA\nmz := \frac{1}{\tau}(\sfA^{n+1}-\sfA^n) \quad \forall n\ge -1.
\end{equation} 
Similarly, let us set $\sfB^n := \frac{1}{\tau^3}(\sfUnp-3\sfUn+3\sfUnm-\sfUnmm)$ 
for all $n \ge \polN$, and
\begin{equation} \label{eq:def_sfBnmz}
\sfB\nmz := \frac12(\sfB^{n+1}+\sfB^n), \quad
\dsfB\nmz := \frac{1}{\tau}(\sfB^{n+1}-\sfB^n) \quad \forall n\ge -1.
\end{equation} \end{subequations}
Notice that $\sfB^{n+1}=\dsfA\nmz$.

\begin{corollary}[Estimates on acceleration] \label{cor:estim_sfA}
Assume the CFL condition~\eqref{eq:CFL}. 
Let $(\sfU^n)_{n\in\polN}$ solve~\eqref{eq:leapfrog} with $\sfU^1=\sfU^0=0$. Let 
$\sfA\nmz$, $\dsfA\nmz$ be defined in~\eqref{eq:def_sfAnmz}, and let 
$\sfB\nmz$, $\dsfB\nmz$ be defined in~\eqref{eq:def_sfBnmz}. 
Recall that $\theta:=\rho \tau \in (0,1]$. The following holds:
\begin{subequations} \begin{align}
\sum_{n\in\polN} \tau \big( \mu_0 \|\dsfA\nmz\|_\Omega^2 + \|\nabla \sfA\nmz\|_\Omega^2 \big) e^{-2\rho t^n} 
& \le C_\sfA(\theta) \mu_0^{-1} \frac{1}{\rho^2} \int_0^{+\infty} \|\ddot f(t)\|_\Omega^2 e^{-2\rho t} dt, \label{eq:bnd_sfA} \\
\sum_{n\in\polN} \tau \big( \mu_0 \|\dsfB\nmz\|_\Omega^2 + \|\nabla \sfB\nmz\|_\Omega^2 \big) e^{-2\rho t^n} 
& \le C_\sfB(\theta) \mu_0^{-1} \frac{1}{\rho^2} \int_0^{+\infty} \|\dddot f(t)\|_\Omega^2 e^{-2\rho t} dt, \label{eq:bnd_sfB}
\end{align} \end{subequations}
with $C_\sfA(\theta):=\frac23 C_\sfX(\theta)(1+e^{2\theta})$,
$C_\sfB(\theta):=\frac{11}{20} C_\sfX(\theta)(2+e^{2\theta})$,
and $C_\sfX(\theta)$ defined in Lemma~\ref{lem:damped_leapfrog}.
\end{corollary}

\begin{proof}
(1) Proof of~\eqref{eq:bnd_sfA}. We observe that the sequence $(\sfA^n)_{n\in\polN}$ solves
the generic leapfrog scheme~\eqref{eq:leapfrog_X} with $\sfA^0=\sfA^{-1}=0$ and right-hand side
\[
\sfG^n := \frac{\sfF^{n+1}-2\sfF^n+\sfF^{n-1}}{\tau^2}
= \frac{1}{\tau} \int_{t^{n-1}}^{t^{n+1}} \psi^n(t) \ddot f(t)dt,
\]
where $\psi^n(t)$ denotes the hat basis function in time having support in $[t^{n-1},t^{n+1}]$
and satisfying $\psi^n(t^n)=1$. Since $\int_{t^{n-1}}^{t^{n+1}} \psi_n(t)^2dt=\frac23\tau$, a Cauchy--Schwarz inequality shows that
\[
\|\sfG^n\|_\Omega^2 \le \frac23 \frac{1}{\tau}  \int_{t^{n-1}}^{t^{n+1}} \|\ddot f(t)\|_\Omega^2 dt.
\]
Letting $E\nmz_\sfA:=m_{h\tau}(\dsfA\nmz,\dsfA\nmz)+\|\nabla \sfA\nmz\|_\Omega^2$ and since $m_{h\tau}(\dsfA\nmz,\dsfA\nmz) \ge \mu_0 \|\dsfA\nmz\|_\Omega^2$ owing to~\eqref{eq:mht_L2}, Lemma~\ref{lem:damped_leapfrog} gives
\begin{align*}
\sum_{n\in\polN} \tau \big( \mu_0 \|\dsfA\nmz\|_\Omega^2 + \|\nabla \sfA\nmz\|_\Omega^2 \big) e^{-2\rho t^n}
&\le \sum_{n\in\polN} \tau E_\sfA\nmz \rev{e^{-2\rho t^n}} \\
&\le C_\sfX(\theta) \mu_0^{-1} \frac{1}{\rho^2} \sum_{n\in\polN} \tau \|\sfG^n\|_\Omega^2 \tau e^{-2\rho t^n} \\
&\le \frac23 C_\sfX(\theta) \mu_0^{-1} \frac{1}{\rho^2} \sum_{n\in\polN} e^{-2\rho t^n} \int_{t^{n-1}}^{t^{n+1}} \|\ddot f(t)\|_\Omega^2 dt.
\end{align*}
Since $\theta=\rho\tau\le 1$, we observe that
\[
e^{-2\rho t^n} \int_{t^{n-1}}^{t^{n+1}} \|\ddot f(t)\|_\Omega^2 dt \le 
\int_{J_{n-1}} \|\ddot f(t)\|_\Omega^2 e^{-2\rho t} dt 
+ e^{2\theta} \int_{J_{n}} \|\ddot f(t)\|_\Omega^2 e^{-2\rho t} dt.
\] 
This completes the proof of~\eqref{eq:bnd_sfA}.

(2) The proof of~\eqref{eq:bnd_sfB} is similar and is only sketched.
We observe that the sequence $(\sfB^n)_{n\in\polN}$ solves
the generic leapfrog scheme~\eqref{eq:leapfrog_X} with $\sfB^0=\sfB^{-1}=0$ and right-hand side
\[
\sfG^n := \frac{\sfF^{n+1}-3\sfF^n+3\sfF^{n-1}-\sfF^{n-2}}{\tau^3}
= \frac{1}{\tau^2} \int_{t^{n-2}}^{t^{n+1}} \phi^n(t) \ddot f(t)dt,
\]
with $\phi_n(t):=\psi_n(t)-\psi_{n-1}(t)$. Since $\phi_n$ is skew-symmetric around
$t\nmzm:=\frac12(t^{n-1}+t^n)$, we infer using Fubini's theorem that
\begin{align*}
\int_{t^{n-2}}^{t^{n+1}} \phi^n(t) \ddot f(t)dt & = \int_{t^{n-\frac12}}^{t^{n+1}} \phi^n(t) \big(\ddot f(t)-\ddot f(\tilde t)\big) dt \\
&= \int_{t^{n-\frac12}}^{t^{n+1}} \phi^n(t) \bigg( \int_{\tilde t}^t \dddot f(s) ds \bigg) dt \\
&= \int_{t^{n-2}}^{t^{n+1}} \dddot f(s) \xi(s) ds, \qquad \xi(s):=\int_{\tilde s}^{t^{n+1}} \!\!\!\!\phi_n(t) dt,
\end{align*}
with $\tilde t:= 2t^{n-\frac12}-t$ for all $t\in [t^{n-\frac12},t^{n+1}]$ and
$\tilde s:=t^{n-\frac12}+|s-t^{n-\frac12}|$ for all $s\in [t^{n-2},t^{n+1}]$. 
A direct calculation shows that
\[
\xi(s) = \begin{cases}
\int_s^{t^{n+1}} \frac{t^{n+1}-t}{\tau}dt = \frac{\tau}{2} \big( \frac{t^{n+1}-s}{\tau} \big)^2&
s\in [t^n,t^{n+1}],\\
\frac{\tau}{2} + \int_s^{t^n} 2\frac{t-t\nmzm}{\tau}dt=\frac{3\tau}{4}-\tau\big( \frac{s-t\nmzm}{\tau} \big)^2&
s\in [t\nmzm,t^n].
\end{cases}
\]
This implies that
\[
\int_{t^{n-2}}^{t^{n+1}} \xi(s)^2ds = 2\int_{t^{n-\frac12}}^{t^{n+1}} \xi(s)^2ds
=2\bigg( \int_{t^{n-\frac12}}^{t^n} \xi(s)^2ds + \int_{t^n}^{t^{n+1}} \xi(s)^2ds
\bigg) = \frac{11}{20}\tau^3,
\]
since $\int_{t^{n-\frac12}}^{t^n} \xi(s)^2ds=\tau^3\int_0^{\frac12}(\frac34-u^2)^2du = \frac{9\tau^3}{40}$ and $\int_{t^n}^{t^{n+1}} \xi(s)^2ds =\frac{\tau^3}{4}\int_0^1u^4du=\frac{\tau^3}{20}$.
Invoking the Cauchy--Schwarz inequality, we infer that
\[
\|\sfG^n\|^2_\Omega \le \frac{1}{\tau} \frac{11}{20} \int_{t^{n-2}}^{t^{n+1}} \|\dddot f(t)\|_\Omega^2 dt.
\]
The conclusion is now straightforward.
\end{proof}

\subsection{\rev{A priori error estimates}}
\label{app:a_priori}

For $m \in \mathbb N$, we say that a function $\phi \in L^2(\Omega)$ belongs to $H^m(\calT_h)$
if $\phi|_K \in H^m(K)$ for all $K \in \calT_h$, and we introduce the semi-norm
$|\phi|_{H^m(\calT_h)}^2 := \sum_{K \in \calT_h} |\phi|_K|_{H^m(K)}^2.$
For dimensional consistency, we set $\kappa:=\max(\ell_\Omega,\rho^{-1})$.
For a semi-norm $|{\cdot}|_Y$ equipping the space $Y$ composed of functions defined over $\Omega$,
we use the shorthand notation $|v|_{L^2_\rho(Y)}^2:=\int_0^{+\infty} |v(t)|_Y^2 e^{-2\rho t}dt$.
Let $\pi_h\upE:V\rightarrow V_h$ denote the Riesz elliptic projector such that, for all $v\in V$, the discrete function $\pi_h\upE(v)\in V_h$ is uniquely defined by requiring that  
$(\nabla (\pi_h\upE(v)-v),\nabla w_h)_\Omega=0$ for all $w_h\in V_h$. Recall that
$\|\pi_h\upE(v)-v\|_\Omega \lesssim \kappa h^k |v|_{H^{k+1}(\calT_h)}$ and 
$\|\nabla(\pi_h\upE(v)-v)\|_\Omega \lesssim h^k |v|_{H^{k+1}(\calT_h)}$ 
(a sharper $L^2$-estimate is not needed). Moreover,
we have the $H^1$-stability properties 
$\|\pi_h\upE(v)\|_\Omega \lesssim \ell_\Omega |v|_V$ and $|\pi_h\upE(v)|_V \le |v|_V$.

\begin{theorem}[Damped energy-norm a priori error estimate on $(u-\uht)$]
\label{th:damped_apriori}
Assume the CFL condition~\eqref{eq:CFL}.  Assume that 
$C_{\textup{tim}}(u):= \sum_{r\in\{3{:}4\}}\kappa^{r-2}|u^{(r)}|_{L^2_\rho(V)}$ and
$C_{\textup{spa}}(u):= \sum_{r\in\{0{:}3\}}\kappa^r |u^{(r)}|_{L^2_\rho(H^{k+1}(\calT_h))}$
are bounded. The following holds:
\begin{equation}
\calE_\rho^2(u-\uht) \lesssim \tau^4 C_{\textup{tim}}(u)^2 + h^{2k} C_{\textup{spa}}(u)^2. 
\end{equation}
\end{theorem}

\begin{proof}
We consider the 
sequence $\sfuh:=(\sfuh^n)_{n\in\polN}$ such that 
$\sfuh^n:=\pi_h\upE(u(t^n))$ for all $n\in\polN$. We write
\[
u-\uht = (u-\pi_h\upE(u)) + (\pi_h\upE(u)-R(\sfuh)) + R(\sfuh-\sfU),
\]
and bound the damped energy norm of the three terms on the right-hand side. 

(1) The approximation properties of the elliptic projector and $\ell_\Omega\le\kappa$ give
\[
\calE_\rho(u-\pi_h\upE(u)) \lesssim h^k \big( \kappa |\dot u|_{L^2_\rho(H^{k+1}(\calT_h))}
+ |u|_{L^2_\rho(H^{k+1}(\calT_h))}\big).
\]

(2) Invoking Lemma~\ref{lem:etaf} and the $H^1$-stability of the elliptic projection gives
\[
\calE_\rho(\pi_h\upE(u)-R(\sfuh)) \lesssim \tau^2 \big( \ell_\Omega |u^{(3)}|_{L^2_\rho(V)} +
\tau |u^{(3)}|_{L^2_\rho(V)} \big) \le \tau^2 \kappa |u^{(3)}|_{L^2_\rho(V)},
\]
since $\tau \le \rho^{-1} \le \kappa$.

(3) We observe that $\sfX:=\sfuh-\sfU$ solves the leapfrog scheme with $\sfX^0=\sfX^{-1}=0$ and right-hand side 
\[
\sfG^n := \sfG_1^n + \sfG_2^n :=
\pi_h\upE\big( (\rev{D^2_\tau}u)^n - \ddot u(t^n) \big) + \big( \pi_h\upE(\ddot u(t^n))-\ddot u(t^n) \big) \quad \forall n\in\polN,
\]
with $(\rev{D^2_\tau}u)^n := \frac{1}{\tau^2}\big(u(t^{n+1})-2u(t^n)+u(t^{n-1})\big)$. 
Combining the results of Lemma~\ref{lem:stab} and Lemma~\ref{lem:damped_leapfrog} gives
\[
\calE_\rho^2\big(R(\sfuh-\sfU) \big) \lesssim \sum_{n\in\polN} \tau E_\sfX\nmz e^{-2\rho t^n}
\lesssim \frac{1}{\rho^2} \sum_{n\in\polN} \tau \big(\|\sfG_1^n\|_\Omega^2
+\|\sfG_2^n\|_\Omega^2\big) e^{-2\rho t^n}.
\]
Proceeding as in the proof of Lemma~\ref{lem:etaf} and using the $H^1$-stability of 
the elliptic projection, we infer that 
$\|\sfG_1^n\|_\Omega^2 \lesssim \tau^3 \ell_\Omega^2 \int_{t^{n-1}}^{t^{n+1}} |u^{(4)}(t)|_V^2dt$. 
Moreover, using the standard estimate $\tau|\psi(t^n)|_Y^2\lesssim \int_{J_n}|\psi(t)|_Y^2dt +
\tau^2\int_{J_n} |\dot\psi(t)|_Y^2dt$, we infer that
\[
\tau \|\sfG_2^n\|_\Omega^2 \lesssim h^{2k}\ell_\Omega^2\int_{J_n}|u^{(2)}(t)|_{H^{k+1}(\calT_h)}^2dt
+ h^{2k}\ell_\Omega^2 \tau^2 \int_{J_n}|u^{(3)}(t)|_{H^{k+1}(\calT_h)}^2dt.
\]
Putting the above two bounds together and taking the square root gives
\[
\calE_\rho\big(R(\sfuh-\sfU) \big) \lesssim \kappa^2\tau^2 |u^{(4)}|_{L^2_\rho(V)} + \kappa^2 h^k\big(|u^{(2)}|_{L^2_\rho(H^{k+1}(\calT_h))}+
\kappa|u^{(3)}|_{L^2_\rho(H^{k+1}(\calT_h))} \big).
\]
This completes the proof.
\end{proof}

\begin{corollary}[Damped energy-norm a priori error estimate on $(u-\wht)$]
Under the assumptions of Theorem~\ref{th:damped_apriori}, and provided 
$C_{\textup{tim}}(f):= \sum_{r\in\{2{:}3\}}\kappa^{r-2}\|f^{(r)}\|_{L^2_\rho(L^2)}$
is bounded, the following holds:
\begin{equation}
\calE_\rho^2(u-\wht) \lesssim \tau^4 \big( C_{\textup{tim}}(u)^2 + \kappa^2 C_{\textup{tim}}(f)^2 \big) + h^{2k} C_{\textup{spa}}(u)^2. 
\end{equation}
\end{corollary}

\begin{proof}
Combine Theorem~\ref{th:damped_apriori} with Lemma~\ref{lem:bnd_eta}.
\end{proof}

\begin{remark}[Sharper a priori estimates]
At several places in the above proof, we used suboptimal estimates, like $\tau \leq \kappa$. 
Sharper estimates can be derived. They lead to the same convergence rates, but with higher
powers of $\tau$ multiplying some of the higher time-derivatives of $u$ (and $f$). 
\end{remark}

\bibliographystyle{amsplain}
\bibliography{biblio}

\end{document}